\documentclass{amsart}
\usepackage{lineno}
\usepackage{amssymb}
\usepackage{amsmath}
\usepackage{amsthm}
\usepackage{mathrsfs}
\usepackage{bbm}
\usepackage{cite}
\usepackage{color}
\usepackage{enumerate}

\usepackage{colortbl}

\usepackage[colorlinks,hypertexnames=false]{hyperref}

\usepackage{changes}

\usepackage{mathtools}
\usepackage{extarrows}
\usepackage{setspace}

\hypersetup{
	colorlinks=true,
	linkcolor=black,
	filecolor=black,
	urlcolor=black,
	citecolor=black,
}

\newtheorem{thm}{Theorem}[section]
\newtheorem{cor}[thm]{Corollary}
\newtheorem{defn}[thm]{Definition}
\newtheorem{lem}[thm]{Lemma}
\newtheorem{prop}[thm]{Proposition}
\newtheorem{exa}[thm]{Example}
\newtheorem{rmk}[thm]{Remark}

\newtheorem{ass}[thm]{Assumption}

\newtheorem{proposition}[thm]{Proposition}

\newtheorem{remark}[thm]{Remark}
\newtheorem{example}[thm]{Example}

\newcommand{\End}{\mathop{\mathrm{End}}}
\numberwithin{equation}{section}

\begin{document}

\title[A new approach to slice analysis via slice topology]{A new approach to slice analysis via slice topology}
\author{Xinyuan Dou}
\email[Xinyuan Dou]{douxy@mail.ustc.edu.cn}
\address{Department of Mathematics, University of Science and Technology of China, Hefei 230026, China}
\author{Ming Jin}
\email[M.~Jin]{jinming@fudan.edu.cn}
\address{School of Mathematical Sciences, Fudan University, Shanghai 200433, People’s Republic of China}
\author{Guangbin Ren}
\email[Guangbin Ren]{rengb@ustc.edu.cn}
\address{Department of Mathematics, University of Science and Technology of China, Hefei 230026, China}
\author{Irene Sabadini}
\email[Irene Sabadini]{irene.sabadini@polimi.it}
\address{Dipartimento di Matematica, Politecnico di Milano, Via Bonardi, 9, 20133 Milano, Italy}
\thanks{This work was supported by the NNSF of China (11771412)}


\begin{abstract}
In this paper we summarize some known facts on slice topology in the quaternionic case, and we deepen some of them by proving new results and discussing some examples. We then show, following \cite{Dou2020004}, how this setting allows us to generalize slice analysis to the general case of functions with values in a real left alternative algebra, which includes the case of slice monogenic functions with values in Clifford algebra. Moreover, we further extend slice analysis, in one and several variables, to functions with values in a Euclidean space of even dimension. In this framework, we study the domains of slice regularity, we prove some extension properties and the validity of a Taylor expansion for a slice regular function.
\end{abstract}

\maketitle

\section{Introduction}

Quaternions are a kind of hypercomplex numbers first described by Hamilton in 1843. With the development of the theory of holomorphic functions in complex analysis and its generalizations to higher dimensional cases, similar theories for quaternions were established. The most well-known function theory in quaternionic analysis, was initiated by  Gr. Moisil and R. Fueter \cite{Fueter1934001} who, with his school, greatly contributed to the development of the theory. The holomorphy in quaternionic analysis is defined via the so-called Cauchy-Riemann-Fueter equation
\begin{equation*}	\frac{\partial{f}}{\partial{x_1}}+i\frac{\partial{f}}{\partial{x_2}}+j\frac{\partial{f}}{\partial{x_3}}+k\frac{\partial{f}}{\partial x_4}=0,
\end{equation*}
and has been widely studied, see e.g. \cite{MR2089988, MR1096955, MR2369875,Sudbery1979001}.

Since the class of functions considered by Moisil and Fueter does not contain neither the identity function nor any other monomial function  $f(q)=q^n$, C. G. Cullen \cite{Cullen1965001} considered another class, using the notion of intrinsic functions introduced by Rinehart in \cite{Rinehart1960001}.
Inspired by Cullen's approach, in  \cite{Gentili2007001} Gentili and Struppa started the study of a new theory of quaternionic functions which are holomorphic in a suitable sense. This theory includes convergent power series of the form $\sum_{n\in\mathbb{N}}q^n a_n$ and it is nowadays known as slice quaternionic analysis; it has been widely studied in the past fifteen years, see \cite{Gentili2014001,Colombo2010003,Colombo2015001} and also the books \cite{Colombo2011001B,Gentili2013001B,Alpay2016001B}. Slice analysis has been extended to octonions \cite{Gentili2010001}, Clifford algebras \cite{Colombo2009002} and real alternative $*$-algebra \cite{Ghiloni2011001,Ghiloni20171002}. The richness of slice analysis in one variable makes it natural to look for generalization to several variables and in fact the quaternonic case is considered in \cite{Colombo2012002}, the case of Clifford algebras in \cite{Ghiloni2012001}, the one of octonions in \cite{Ren2020001} and finally, the real alternative $*$-algebras case in \cite{Ghiloni2020001}. The most general setting of Euclidean spaces is treated in \cite{Dou2020004} and in this paper.

A fundamental role in slice quaternionic analysis is played by the so-called representation formula, see \cite{Colombo2009Structure,Colombo2009001}. The formula allows to extend many known results in complex analysis to slice quaternionic analysis, e.g. the quaternionic power series expansion \cite{Gentili2012001,Stoppato2012001}, geometric function theory \cite{Ren2017001,Ren2017002,Wang2017001,Xu2018001}, results in quaternionic Schur analysis \cite{Alpay2012001} and in quaternionic operator theory \cite{Alpay2015001,MR3887616,MR3967697}. However, the representation formula just works on axially symmetric s-domains (see \cite{Colombo2009001} for the definition). This causes difficulties while considering more general, non necessarily axially symmetric, domains.

In the recent paper \cite{Dou2020001}, we generalize the representation formula to non-axially symmetric sets.
To this end, a crucial tool is the use of a suitable topology on $\mathbb{H}$, finer than the Euclidean one, called slice topology and denoted by $\tau_s$. In this paper we deepen the study of the slice topology also proving that it is not metrizable.

One can consider functions defined on open sets in this topology, hence enlarging the class of slice regular functions.  It is important to note that, in this framework, the conditions under which analytic continuation for slice regular functions is possible are weaker than the conditions in use with the Euclidean topology. This fact makes it easier the study of analogues of the Riemann domains (Riemann domains over $(\mathbb{H},\tau_s)$) and generalized manifolds over $(\mathbb{H},\tau_s)$. In the slice topology the class of `domains of holomorphy' includes the axially symmetric open sets and also the so-called hyper-$\sigma$-polydiscs that we shall introduce in Section \ref{sc-dsr}.

We then show how the notion of slice regularity can be given for functions with values in a finite-dimensional left alternative algebra, see Section \ref{sc-srla}, and also for functions with values in the Euclidean space $\mathbb R^{2n}$. It is important to note that slice regularity in this paper is meant in the original sense of Gentili and Struppa, see \cite{Gentili2007001} (we shall refer to it as the weak slice regularity) and not in the sense of Ghiloni and Perotti \cite{Ghiloni2020001} (regularity in their sense may also be called strong slice regularity). Moreover, we also consider the several variables case thus giving rise to results of a more general validity than those one available in the literature.
Our approach is rather new and it is spread in various papers, see \cite{Dou2020001, Dou2020003, Dou2020004}, and the main purpose of this work is to give a unified overview of the main ideas, complemented with some other new results and examples.

The paper is organized in seven sections, besides this introduction. In Section 2 we introduce the slice topology in the quaternionic case and we prove some of its properties among which the non-metrizability. In Section 3 we discuss the definition of quaternionic slice regular functions, some main properties and examples. In Section 4 we consider the generalization to the case of functions with values in a  left alternative algebra and defined on the quadratic cone whereas in Section 5 we come to the case of functions with values in a real Euclidean space, even dimensional, and defined on the so-called slice-cones. In this extremely general setting we can prove, in Section 6, an extension result. The domains of slice regularity are discussed in Section 7, while a Taylor formula is proved in Section 8.

\section{Slice topology}
In this section, we define the so called slice topology $\tau_s$ on $\mathbb H$, originally introduced in \cite{Dou2020001}, we state some of its properties and we discuss some examples. We shall show, in particular, that slice topology has some special features near the real axis and this fact has important consequences when considering connectedness.

We recall that the algebra of real quaternions $\mathbb H$ consists of elements of the form $q=x_0+x_1i+x_2j+x_3k$ where the imaginary units $i$, $j$, $k$ satisfy $i^2=j^2=k^2=-1$, $ij=-ji=k$, $jk=-kj=i$, $ki=-ik=j$.

The set of imaginary units in $\mathbb{H}$ is defined by
\begin{equation*}
	\mathbb{S}:=\{q\in\mathbb{H}\ :\ q^2=-1\}.
\end{equation*}

The slice topology that we introduce below is designed around the notion of slice regularity, see Definition \ref{slreg}, and is based on the following:
\begin{defn}\label{df-so}
	A subset $\Omega$ of $\mathbb{H}$ is called slice-open, if
	\begin{equation*}
		\Omega_I:=\Omega\cap\mathbb{C}_I
	\end{equation*}
 is open in $\mathbb{C}_I$ for any $I\in\mathbb{S}$.
\end{defn}
The following result is easily proved with classical arguments:
\begin{lem} The family
	$$\tau_s(\mathbb{H}):=\{\Omega\subset\mathbb{H}:\Omega\ \mbox{is slice-open}\}$$ defines a topology of $\mathbb{H}$.
\end{lem}

\begin{defn}
	We call $\tau_s(\mathbb{H})$ the slice topology on $\mathbb{H}$. Open sets, connected sets and paths in the slice topology are called  slice-open sets, { slice-connected sets} and slice-paths. A domain in the slice topology is called {\em slice topology-domain} or, in short, {\em st-domain}.
\end{defn}
A similar terminology will be used for all the other notions in the slice topology. We made an exception and we do not use the term slice-domain to denote a domain in the slice topology, since this notion is already used in the literature to denote something different. In fact a set  $\Omega$ in $\mathbb{H}$ is called {\em classical slice domain}, in short {\em s-domain}, if it is a domain in the Euclidean topology such that $\Omega_\mathbb{R}:=\Omega\cap \mathbb R\neq\varnothing$, and $\Omega_I$  is a domain in $\mathbb{C}_I$ for any $I\in\mathbb{S}$.

It is immediate from the definition that for any $I\in\mathbb{S}$, the subspace topology $\tau_s(\mathbb{C}_I\backslash\mathbb{R})$ of $\tau_s(\mathbb{H})$ coincides the Euclidean topology $\tau(\mathbb{C}_I\backslash\mathbb{R})$. However, $\tau_s$ is quite different from Euclidean topology near $\mathbb{R}$ as  Example 3.2 in \cite{Dou2020001} shows.

The peculiarity of the topology $\tau_s$ near $\mathbb{R}$ appears, in particular, in the notion of connectedness, so we need to introduce another useful notion.

\begin{defn}
Let	$\Omega\subset\mathbb{H}$ and let $\Omega_{\mathbb{R}}:=\Omega\cap\mathbb{R}$. The set $\Omega$ is called real-connected if $\Omega_{\mathbb{R}}\not=\varnothing$ and $\Omega_{\mathbb{R}}$ is connected in $\tau(\mathbb{R})$ or if $\Omega_{\mathbb{R}}=\varnothing$.
\end{defn}

For any $q\in\mathbb{H}\backslash\mathbb{R}$, there is $r\in\mathbb{R}_+$ such that the ball $B_I(q,r)$ in $\mathbb{C}_I$ does not intersect $\mathbb{R}$. Note that for such $r>0$  $B_I(q,r)\cap\mathbb{R}=\varnothing$ is connected in $\mathbb{R}$. It is then clear that the slice topology has a basis of real-connected sets at any point outside $\mathbb{R}$, and the following result implies that slice topology has a real-connected basis also near $\mathbb{R}$.

\begin{prop}\label{pr-sor}
	For any slice-open set $\Omega$ in $\mathbb{H}$ and $q\in\Omega$, there is a real-connected st-domain $U\subset\Omega$ containing $q$.
\end{prop}
It is useful to explain that such a $U$ can be constructed as the slice-connected component of the set
$$(\Omega\backslash\Omega_{\mathbb{R}})\cup A$$ containing $q$.
 Here when
    $q\in\mathbb{R}$,  we take $A$ to be the connected component of $\Omega_{\mathbb{R}}$ containing $q$ in $\mathbb{R}$ and when $q\not\in\mathbb{R}$ we set $A:=\varnothing$.

Now we describe slice-connectedness of real-connected slice-domains by using suitable slice-paths.

\begin{defn}
	A path $\gamma$ in $(\mathbb{H},\tau)$ is called on a slice, if $\gamma\subset\mathbb{C}_I$ for some $I\in\mathbb{S}$.
\end{defn}

One can prove, see \cite[Proposition 3.6]{Dou2020001}, that any path on a slice is a slice-path. Moreover we have, see \cite[Proposition 3.8]{Dou2020001}:

\begin{prop}\label{pr-rcsd}
	For each real-connected st-domain $U$, the following assertions hold:
	\begin{enumerate}[\upshape (i)]
		
		\item If $U_\mathbb{R}=\varnothing$, then $U\subset\mathbb{C}_I$ for some $I\in\mathbb{S}$.
		
		\item If $U_\mathbb{R}\neq\varnothing$, then for each $q\in U$ and $x\in U_\mathbb{R}$, there is a path on a slice from $q$ to $x$.
		
		\item $U_I$ is a domain in $\mathbb{C}_I$ for each $I\in\mathbb{S}$.
		
		\item For each $p,q\in U$, there are two paths on a slice $\gamma_1,\gamma_2$ in $U$ such that $\gamma_1(1)=\gamma_2(0)$ and $\gamma_1\gamma_2$ is a slice-path from $p$ to $q$.
		
	\end{enumerate}
\end{prop}

By Proposition \ref{pr-sor} and \ref{pr-rcsd} (iv), we deduce the next important result:

\begin{prop}
	The topological space $(\mathbb{H},\tau_s)$ is connected, local path-connected and path-connected.
\end{prop}

\begin{cor}\label{co-dsd}
	A set $\Omega\subset\mathbb{H}$ is an st-domain if $\Omega_\mathbb{R}\neq\varnothing$ and $\Omega_I$ is a domain in $\mathbb{C}_I$ for any $I\in\mathbb{S}$.
\end{cor}

	By Corollary \ref{co-dsd}, any s-domain is an st-domain. Therefore  the notion of st-domain is a generalization of the notion of s-domain. However, the converse is not true since not every st-domain $\Omega$ is an s-domain, not even when $\Omega$ is a domain in $\mathbb{H}$, as Example 3.13 in \cite{Dou2020001} shows. In that example we consider the set
	\begin{equation*}
		\Omega:=\mathbb B(0,2)\cup \mathbb B(6,2)\cup \left\{q\in\mathbb{H}:\mbox{dist}(q-I,[0,6])<\frac 12\right\},
	\end{equation*}
which is a domain in $\tau(\mathbb{H})$.
Let us fix
 $I\in\mathbb{S}$
	and let $J\in\mathbb{S}$ be such that $J\bot I$. Then
	\begin{equation*}
		\Omega_J=\mathbb B_J(0,2)\cup\mathbb B_J(6,2)
	\end{equation*}
	 is not connected in $\tau(\mathbb{C}_J)$, and as a consequence, $\Omega$ is not an s-domain. However, $\Omega$ is slice-connected, because any point in $\Omega$ can be connected to $0$ or $6$ by a path in a slice, and $0$ can be connected to $6$ by a path in $\mathbb{C}_I$. So $\Omega$ is an st-domain since $\Omega_J$ is open in $\mathbb{C}_J$ for any $J\in\mathbb{S}$.

We note that in the quaternionic case, besides the Euclidean and the slice topology, we can also consider the topology introduced in \cite{Gentili2012001} and based on the so-called $\sigma$-distance which is
defined by
\begin{equation*}
	\sigma(q,p):=\left\{\begin{aligned}
		&|q-p|,&&\qquad\exists\ I\in\mathbb{S},\ s.t.\ p,q\in\mathbb{C}_I,
		\\&\sqrt{(Re(q-p))^2+|Im(q)|^2+|Im(p)|^2},&&\qquad {\rm otherwise}.
	\end{aligned}\right.
\end{equation*}
The $\sigma$-balls centered at  $q\in\mathbb{H}$, namely the balls according to the $\sigma$-distance, are defined by:
\begin{equation*}
	\Sigma(p,r):=\{q\in\mathbb{H}:\sigma(p,q)<r\}.
\end{equation*}
The slice topology is finer than the topology $\tau_\sigma$ induced by the  $\sigma$-distance. In fact, the relations between {
the topologies} $\tau, \tau_\sigma,\tau_s$ is described in the next result:
\begin{prop}
The topologies $\tau, \tau_\sigma,\tau_s$ are such that
$$
\tau\subsetneq \tau_\sigma\subsetneq\tau_s.
$$
\end{prop}

\begin{proof} The inclusion $\tau\subsetneq \tau_\sigma$ is immediate. It is { also immediate} to check that $\tau_\sigma\subset\tau_s$. To prove that the second inclusion is strict, we need to show a set $\Omega\in\tau_s\backslash\tau_\sigma$. To construct such a set we fix $I\in\mathbb{S}$ and we consider the slice-open set  $\Omega$ in $\mathbb H$ defined by
	\begin{equation}\label{eq-ob}
		\Omega:=\bigcup_{J\in\mathbb{S}}\Omega_J,
	\end{equation}
	where
	\begin{equation*}
		\Omega_J:=\left\{
		\begin{aligned}
			&\{x+yJ\in\mathbb{C}_J:x^2+\frac{y^2}{\mbox{dist}(J,\mathbb{C}_I)}<1\},\qquad&&J\neq\pm I,
			\\&\{x+yJ\in\mathbb{C}_J:x^2+y^2<1\}, &&J=\pm I,
		\end{aligned}\right.
	\end{equation*}
	and $\mbox{dist}(J,\mathbb{C}_I)$ is the Euclidean distance in $\mathbb{H}$ from $J$ to $\mathbb{C}_I$.
The set $\Omega$ is evidently a slice-open set. Since $\mathbb{H}\backslash\Omega\supset\mathbb{C}_J\backslash\Omega_J$, we have
\begin{equation}\label{eq-s0}
	\sigma(0,\mathbb{H}\backslash\Omega)\le \sigma(0,\mathbb{C}_J\backslash\Omega_J)= \mbox{dist}_{\mathbb{C}_J}(0,\mathbb{C}_J\backslash\Omega_J),
\end{equation}
where $\mbox{dist}_{\mathbb{C}_J}(0,\mathbb{C}_J\backslash\Omega_J)$ is the Euclidean distance in $\mathbb{C}_J$ between $0$ and $\mathbb{C}_J\backslash\Omega_J$. Since
\begin{equation}\label{eq-lj}
	\lim_{J\rightarrow I, J\neq I}\left[\mbox{dist}_{\mathbb{C}_J}(0,\mathbb{C}_J\backslash\Omega_J)\right]=0,
\end{equation}
by \eqref{eq-s0} and \eqref{eq-lj}, we have $$\sigma(0,\mathbb{H}\backslash\Omega)=0.$$ Hence, $0$ is not an interior point in $\Omega$ in the topology $\tau_\sigma$. We deduce that $\Omega$ is not open in $\tau_\sigma$.
\end{proof}

Another peculiarity of the slice topology is described in the next proposition:
\begin{proposition}
The topology $\tau_s$ is not a metrizable topology.
\end{proposition}
\begin{proof}
Suppose that $\tau_s$ is induced by a metric $d_s$. Then the set
	\begin{equation*}
		\mathcal{O}:=\{q\in\mathbb{H}:d_s(q,0)<1\}
	\end{equation*}
	is a slice-open set in $\tau_s$. For any $\epsilon>0$ we introduce the set
	\begin{equation}
\mathbb{S}_\epsilon:=\{I\in\mathbb{S}:B_I(0,\epsilon)\subset\mathcal{O}\},
	\end{equation}
and we denote by $\left|\mathbb{S}_\epsilon\right|$ its cardinality.
We claim that $\left|\mathbb{S}_\epsilon\right|$ is finite.
	Assume the contrary. Then there are $J_1,...,J_k,...\in \mathbb{S}_\epsilon$ with $J_\imath\neq\pm J_\jmath$. Set
	\begin{equation*}
		U[I]=\begin{cases}
			\left\{x+yJ_k:{x^2}+\frac{y^2}{(1/k^2)}<1\right\},& I=\pm J_k,\mbox{ for some } k\in\mathbb{N}_+,
			\\\left\{x+yI:{x^2}+{y^2}<1\right\},& \mbox{otherwise}.
		\end{cases}
	\end{equation*}
	The set
	\begin{equation*}
		U:=\bigcup_{I\in\mathbb{S}} U[I]
	\end{equation*}
	is a slice-open set. However
	\begin{equation*}
		d_s(0,\mathbb{H}\backslash U)\le \lim_{k\rightarrow +\infty} d_s(0,\mathbb{C}_{J_k}\backslash U[J_k])
		=\lim_{k\rightarrow +\infty}\frac{1}{k}=0,
	\end{equation*}
	so $0$ is not an interior point of $U$ in the slice topology, and $U$ is not slice-open, which is a contradiction. Hence $\left|\mathbb{S}_\epsilon\right|$ is finite.
	
	On the other hand, for each $I\in\mathbb{S}$,
	\begin{equation*}
		\mbox{dist}_{\mathbb{C}_I}(0,\mathbb{H}\backslash \mathcal{O}_I)>0,
	\end{equation*}
	since ${\mathcal{O}_I}$ is open in $\mathbb{C}_I$. Then $I\in\mathbb{S}_{\frac{1}{n}}$ for some $n\in\mathbb{N}_+$ and
	\begin{equation*}
		\mathbb{S}=\bigcup_{k\in\mathbb{N}_+}\mathbb{S}_{\frac{1}{k}}.
	\end{equation*}
	Since the cardinality of $\mathbb{S}_{\frac{1}{k}}$ is finite, we deduce that $\mathbb{S}$ is a countable set, which is absurd. Hence $\tau_s$ cannot be induced by any metric.
\end{proof}

\section{Main results and examples}\label{sc-mr}

The definition of slice regular functions recalled below is well known since \cite{Gentili2006001}, but the novelty here is that we work with the slice topology and the functions are defined on slice-open sets in $\mathbb{H}$. We present some results below, whose proof is given in \cite{Dou2020001}, and we discuss some examples that further clarify how our approach with the slice topology gives a richer function theory and allows more general situations.

\begin{defn}\label{slreg}
	Let $\Omega$ be a slice-open set in $\mathbb{H}$. A function $f:\Omega\rightarrow\mathbb{H}$ is called (left) slice regular, if $f_I:=f|_{\Omega_I}$ is left holomorphic for any $I\in\mathbb{S}$, i.e.
	if $f$ is real differentiable and satisfies
	\begin{equation*}
		\frac{1}{2}\left(\frac{\partial}{\partial x}+I \frac{\partial}{\partial y}\right) f(x+yI)=0\qquad\mbox{on}\qquad{\Omega_I},
	\end{equation*}
	where
	\begin{equation*}
		\Omega_I:=\Omega\cap\mathbb{C}_I.
	\end{equation*}
\end{defn}

The following result is known as Splitting Lemma and it is based on writing the values of a quaternionic function by means of two complex-valued functions, and for this reason the result holds also in this framework:

\begin{lem}
	(Splitting Lemma)
	Let $\Omega\in\tau_s(\mathbb{H})$. A function $f:\Omega\rightarrow\mathbb{H}$ is slice regular, if and only if for all $I,J\in\mathbb{S}$ with $I\bot J$, there are two $\mathbb{C}_I$-valued holomorphic functions $F,G: \Omega_I\rightarrow\mathbb{C}_I$ such that $f_I=F+GJ$.
\end{lem}
A result which is classical for holomorphic functions is the identity principle. This result holds also for slice regular functions defined on domains in the Euclidean topology, but when considering the slice topology in $\mathbb H$, its proof is more delicate. We recall its statement here since it is crucial to prove various results:

\begin{thm}
	(Identity Principle)
	Let $\Omega$ be an st-domain, namely a domain in $\tau_s(\mathbb{H})$, and let $f,g:\Omega\rightarrow\mathbb{H}$ be slice regular.  If  $f$ and $g$ coincide on a subset of $\Omega_I$ with an accumulation point in $\Omega_I$ for some $I\in\mathbb{S}$, then $f=g$ on $\Omega$.
\end{thm}
Another crucial result for slice regular functions is the so-called extension formula, see \cite[Theorem 4.2]{Colombo2009001}, which is used to prove the general representation formula \cite[Theorem 3.2]{Colombo2009001} in the class of axially symmetric slice domains. In \cite{Dou2020001}, we have extended this result to a more general setting and we proved the following:

\begin{thm}(Extension Theorem)
	Let $I_1,I_2\in\mathbb{S}$ with $I_1\neq I_2$, $U_1\in\tau(\mathbb{C}_{I_1})$ and $U_2\in\tau(\mathbb{C}_{I_2})$.  If  $f:U_1\cup U_2\rightarrow\mathbb{H}$ is a function such that  $f|_{U_1}$ and $f|_{U_2}$ are holomorphic, then the function $f|_{V^+}$ admits a slice regular extension $\widetilde{f}$ to $V^{+\Delta}\in\tau_s(\mathbb{H})$.
	
	Moreover, for each domain $W$ in $\tau_s(\mathbb{H})$,
	\begin{equation*}
		W\subset V^{+\Delta}, \qquad W\cap V^+\neq\varnothing,
	\end{equation*}
	$\widetilde{f}|_W$ is slice and the unique slice regular extension on $W$  of $f|_{W\cap V^+}$, where
	
	\begin{eqnarray*}
		V^+ &:=& (U_1\cup \mathbb{C}_{I_1}^+)\bigsqcup (U_2\cap\mathbb{C}_{I_2}^+)\bigsqcup (U_1\cap U_2\cap\mathbb{R}),
		\\
		V^\Delta &:=& \{x+y \mathbb{S}: x+yI_1\in U_1, x+yI_2\in U_2,    y\in\mathbb{R},  \ y\ge 0\},
		\\
		V^{+\Delta} &:=& V^+\cup V^\Delta.
	\end{eqnarray*}
\end{thm}
We note that if we consider a disc  $B_I(q,r)\subset\mathbb{C}_I$, $I\in\mathbb{S}$, with center $q\in\mathbb{C}_I$ and radius $r\in\mathbb{R}_+$, and a holomorphic function $f:B_I(q,r)\rightarrow\mathbb{H}$, then $f$ can be uniquely extended  to be  a slice regular function { on the $\sigma$-ball}  $\Sigma(q,r)$.
\smallskip

Another approach to slice regular functions makes use of the notion of slice functions. To recall this notion, we first introduce a notation: for any $\Omega\subset\mathbb{H}$, define a set in $\mathbb{C}$ by
\begin{equation*}
	\Omega_s:=\{x+yi\in\mathbb{C}:x,y\in\mathbb{R},\ \exists\ J\in\mathbb{S}\ |\ x+yJ\in\Omega\}.
\end{equation*}

\begin{defn}
	Let $\Omega\subset\mathbb{H}$ and function $f:\Omega\rightarrow\mathbb{H}$; $f$ is called slice function if there is a  $F:\Omega_s\rightarrow\mathbb{H}^{2\times 1}$ such that
	\begin{equation}\label{eq-fzfx}
		f(x+yI)=(1,\ I)F(x+yi),
	\end{equation}
	for each $x+yI\in\Omega$ with $x,y\in\mathbb{R}$ and $I\in\mathbb{S}$.
	We call $F$ is a stem function of $f$.
\end{defn}
For any $I\in\mathbb{S}$, there is an isomorphism
\begin{align*}
	\mathcal{P}_I:\quad\mathbb{C}\quad\xlongrightarrow[\hskip1cm]{}&\quad \mathbb{C}_I
	\\x+yi\ \shortmid\!\xlongrightarrow[\hskip1cm]{}&\ x+yI,
\end{align*}
and for each path $\gamma:[0,1]\rightarrow\mathbb{C}$, the corresponding path in $\mathbb{C}_I$ is denoted by
\begin{equation*}
	\gamma^I:=\mathcal{P}_I\circ\gamma.
\end{equation*}

Let  $\mathscr{P}(\mathbb{C})$ be the set of paths $\gamma: [0,1]\longrightarrow \mathbb C$ with initial point $\gamma(0)$ in $\mathbb{R}$. We define a subset of $\mathscr{P}(\mathbb{C})$ by setting
\begin{equation*}
	\mathscr{P}(\mathbb{C}^+):=\{\gamma\in\mathscr{P}(\mathbb{C}):\gamma(0,1]\subset\mathbb{C}^+\}.
\end{equation*}
Given $\Omega\subset\mathbb{H}$ and $\gamma\in\mathscr{P}(\mathbb{C})$ we define
\begin{equation*}
	\mathscr{P}(\mathbb{C},\Omega):=\{\delta\in\mathscr{P}(\mathbb{C}):\exists\ {I\in\mathbb{S}},\mbox{ s.t. }\delta^{I}\subset\Omega\}
\end{equation*}
\begin{equation*}
	\mathscr{P}(\mathbb{C}^+,\Omega):=\{\delta\in\mathscr{P}(\mathbb{C}^+):\exists\ {I\in\mathbb{S}},\mbox{ s.t. }\delta^{I}\subset\Omega\}
\end{equation*}
and
\begin{equation*}
	\mathbb{S}(\gamma,\Omega):=\{I\in\mathbb{S}:\gamma^{I}\subset\Omega\}.
\end{equation*}

We now generalize the definition of slice function:
\begin{defn}\label{Def36}
	Let $\Omega\subset\mathbb{H}$. A function $f:\Omega\rightarrow\mathbb{H}$ is called path-slice if there is a function $F:\mathscr{P}(\mathbb{C},\Omega)\rightarrow\mathbb{H}^{2\times 1}$ such that
	\begin{equation*}
		f\circ\gamma^{I}(1)=(1,\ I)F(\gamma),
	\end{equation*}
	for each $\gamma\in\mathscr{P}(\mathbb{C},\Omega)$ and $I\in\mathbb{S}(\gamma,\Omega)$.
	We call $F$ is a path-slice stem function of $f$.
\end{defn}

Path-slice functions are characterised in the next result:

\begin{prop}\label{pr-ps}
	Let $\Omega\subset\mathbb{H}$ and $f:\Omega\rightarrow\mathbb{H}$. Then following statements are equivalent:
	\begin{enumerate}[\upshape (i)]
		
		\item $f$ is a path-slice function.
		
		\item For each $\gamma\in\mathscr{P}(\mathbb{C},\Omega)$, there is an element  $q_\gamma\in\mathbb{H}^{2\times 1}$ such that
		\begin{equation*}
			f\circ\gamma^I(1)=(1,I)q_\gamma\qquad\forall\ I\in{\mathbb{S}(\gamma,\Omega)}.
		\end{equation*}
		
		\item For each $\gamma\in\mathscr{P}(\mathbb{C}^+,\Omega)$, there is an element  $p_\gamma\in\mathbb{H}^{2\times 1}$ such that for each $I\in\mathbb{S}(\gamma,\Omega)$,
		\begin{equation*}
			f\circ\gamma^I(1)=(1,I)p_\gamma.
		\end{equation*}
		
		\item For each $\gamma\in\mathscr{P}(\mathbb{C},\Omega)$ and $I,J,K\in\mathbb{S}(\gamma,\Omega)$ with $J\neq K$,
		\begin{equation*}
			f\circ\gamma^I=(1,I)\left(\begin{matrix}
				1&J\\1&K
			\end{matrix}\right)^{-1}
			\left(\begin{matrix}
				f\circ\gamma^J\\f\circ\gamma^K
			\end{matrix}\right)=\left(\begin{matrix}
(J-K)^{-1}J&(K-J)^{-1}K\\(J-K)^{-1}&(K-J)^{-1}
\end{matrix}\right)\left(\begin{matrix}
				f\circ\gamma^J\\f\circ\gamma^K
			\end{matrix}\right).
		\end{equation*}
		
		\item For each $\gamma\in\mathscr{P}(\mathbb{C})$ and $I,J,K\in\mathbb{S}(\gamma,\Omega)$ with $J\neq K$,
		\begin{equation*}
			\begin{split}
				f\circ\gamma^I=(J-K)^{-1}(Jf\circ\gamma^J-Kf\circ\gamma^K)+I(J-K)^{-1}(f\circ\gamma^J-f\circ\gamma^K).
			\end{split}
		\end{equation*}
		
	\end{enumerate}
\end{prop}
The above definition is indeed a generalization of the notion of slice functions, and in fact in \cite{Dou2020001} we proved:
\begin{prop}
	Every slice function defined on a subset of $\mathbb{H}$ is path-slice.
\end{prop}

The class of path-slice functions also contains the class of slice regular functions:

\begin{thm}
	Every slice regular function defined on an open set in $\tau_s(\mathbb{H})$ is path-slice.
\end{thm}

Note that a slice regular function is not necessarily a slice function, unless one adds hypothesis on the open set of definition, for example that it is an axially symmetric s-domain.

A cornerstone of the future development of slice quaternionic analysis on slice-open sets is the following result, originally proved in \cite{Dou2020001}:

\begin{thm}\label{th-rf}
	(Representation Formula) Let $\Omega\in\tau_s(\mathbb{H})$ and $f:\Omega\rightarrow\mathbb{H}$ be slice regular. For each $\gamma\in\mathscr{P}(\mathbb{C},\Omega)$ and $I,J,K\in\mathbb{S}(\gamma,\Omega)$ with $J\neq K$, we have
	\begin{equation*}
		f\circ\gamma^I=(1,I)
		\begin{pmatrix}
			1&J\\1&K
		\end{pmatrix}^{-1}
		\begin{pmatrix}
			f\circ\gamma^{J}\\f\circ\gamma^{K}
		\end{pmatrix}.
	\end{equation*}
\end{thm}

Another important concept is the one of domain of holomorphy for slice regular functions, namely the notion of domain of slice regularity.

\begin{defn}\label{df-dsr}
A slice-open set $\Omega\subset\mathbb{H}$ is called a domain of slice regularity if there are no slice-open sets $\Omega_1$ and $\Omega_2$ in $\mathbb{H}$ with the following properties.
	\begin{enumerate}[\upshape (i)]
		\item $\varnothing\neq\Omega_1\subset\Omega_2\cap\Omega$.
		\item $\Omega_2$ is slice-connected and not contained in $\Omega$.
		\item For any slice regular function $f$ on $\Omega$, there is a slice regular function $\widetilde{f}$ on $\Omega_2$ such that $f=\widetilde{f}$ in $\Omega_1$.
	\end{enumerate}
Moreover, if there are slice-open sets $\Omega,\Omega_1,\Omega_2$ satisfying (i)-(iii), then we call $(\Omega,\Omega_1,\Omega_2)$ a slice-triple.
\end{defn}
Here the situation is different from complex analysis, since not every slice-open set is a domain of slice regularity. It is true however that the $\sigma$-balls and axially symmetric slice-open sets are particular domains of slice regularity, see Proposition 9.5 and Proposition 9.6 in \cite{Dou2020001}. Below we illustrate an example:

\begin{exa}
The $\sigma$-ball $\Sigma(\frac{I}{2},1)$ is a domain of slice regularity  for any $I\in\mathbb{S}$.
 The function $f:\Sigma(\frac{I}{2},1)\rightarrow\mathbb{H}$ defined by
	\begin{equation*}
	f(q)=\sum_{n\in\mathbb{N}}(q-I/2)^{*2^n},
	\end{equation*}
	does not extend to a slice regular function near any point of the boundary in any slice $\mathbb{C}_J$, $J\in\mathbb{S}$.\\
\end{exa}
This example shows that there exist
 slice regular functions defined on a slice-open set which is not open in the Euclidean topology but it is open in the $\tau_\sigma$ topology. Moreover, there are examples of functions defined on slice-open sets which are neither open in the Euclidean topology nor in $\tau_\sigma$, as the Example \ref{exa314} shows. The example is obtained by further elaborating Example \ref{exa313} which comes from ideas in \cite{Dou2020001}, Section 8.
 \begin{exa}\label{exa313}
 We consider a ray  $\gamma_s: [0, 1)\longrightarrow  \mathbb{C}$,  $s\in[0,1]$ fixed,  by
\begin{equation*}
\gamma_s (t):=\frac{i}{2}+\frac{t}{1-t}e^{i(\frac{\pi }{4}+(s\pi)/2)}.
\end{equation*}
The ray starts from $\frac{i}{2}$ to $\infty$  and the angle between the ray and
the positive real axis
 is $\frac{\pi}{4}+s\frac{\pi}{2}$.

Given a continuous function  $$\varphi:\mathbb{S}\rightarrow[0,1],$$
we define a continuous function $F:\mathbb{S}\times[0,1)\rightarrow\mathbb{H}$ by setting
\begin{equation*}
F(I, s)=\mathcal{P}_I\circ\gamma_{\varphi(I)}(s).
\end{equation*}
The complement of  the image of $F$ is denoted by
\begin{equation*}
\Omega_{\varphi}:=\mathbb{H}\backslash F(\mathbb{S}\times[0,1)).
\end{equation*}
{Fix $J\in\mathbb{S}$.} We now define
\begin{equation}\label{Omegatilde}
	\widetilde{\Omega_\varphi}:=\Omega_\varphi\bigcup\gamma_\varphi[-J].
	\end{equation}
We proved in \cite{Dou2020001} that $\widetilde{\Omega_\varphi}$  is open in the Euclidean topology if $\varphi$ is continuous. However it may be not open if $\varphi$ is not continuous. For example, let $\varphi:\mathbb{S}\rightarrow[0,1]$ be the Dirichlet type function, i.e.
\begin{equation*}
	\varphi(I)=\begin{cases}
		1,& |I-J|\in\mathbb{Q};
		\\ 0,& \mbox{otherwise}.
	\end{cases}
\end{equation*}
Let $I\in\mathbb{S}$ with $I\neq\pm J$ and $|I-J|\in\mathbb{Q}$, where $J$ is fixed. Then $\varphi(I)=1$ and
\begin{equation*}
	\left(\widetilde{\Omega_\varphi}\right)\cap\mathbb{C}_I^+=\mathbb{C}_I^+\backslash \left(\mathcal{P}_I\circ\gamma_1[0,1]\right)\qquad\mbox{and}\qquad \mathcal{P}_I\circ\gamma_0(0,1]\subset\left(\widetilde{\Omega_\varphi}\right)_I.
\end{equation*}
We choose a sequence $\{K_\ell\}_{\ell=1}^{+\infty}\subset\mathbb{S}\backslash\{\pm J\}$, such that $|K_\ell-J|\notin\mathbb{Q}$ and
\begin{equation*}
	\lim_{\ell\rightarrow +\infty} K_\ell=I.
\end{equation*}
Notice that $\varphi(K_\ell)=0$ and
\begin{equation*}
	\widetilde{\Omega_\varphi}\cap\mathbb{C}_{K_\ell}^+=\mathbb{C}^+_{K_\ell}\backslash(\mathcal{P}_{K_\ell}\circ\gamma_0(0,1]).
\end{equation*}
Notice that for each $t\in(0,1]$, we have
\begin{equation*}
	\lim_{\ell\rightarrow +\infty}\mathcal{P}_{K_\ell}\circ\gamma_0(t)=\mathcal{P}_I\circ\gamma_0(t).
\end{equation*}
It is clear that $\mathcal{P}_I\circ\gamma_0(t)$ is not an interior point of $\Omega$ in the Euclidean topology and $\Omega$ is not open in the Euclidean topology.

We now consider the function
 \begin{equation}\label{eq-fp647}
\Psi (z)=\sqrt{2z-J}, \qquad  \forall\  z\in \frac{J}{2}+\mathbb{R}_+,
\end{equation}
where $z$ is a complex variable. This function admits  a unique holomorphic extension $\Psi_s$ on $$\mathbb{C}_J\backslash(\gamma_s[J]\cup\gamma_s[-J]), $$
where
  $$\gamma_s[J]:=\mathcal{P}_J\circ\gamma_s([0,1))$$
 for any $s\in[0,1]$.

The function $\Psi_\varphi:\Omega_\varphi\rightarrow\mathbb{H}$ defined by
	\begin{equation}\label{eq-fp}
	\Psi_\varphi(x+yI):=
		\frac{1-IJ}{2}\Psi_{\varphi(I)}(x+yJ)+\frac{1+IJ}{2}\Psi_{\varphi(I)}(x-yJ),
	\end{equation}
	for $y\ge 0$, is the unique slice regular extension of $\Psi$ to $\Omega_\varphi$. In particular, $$(\Psi_\varphi)_J=\Psi_{\varphi(J)}.$$

Moreover, $\widetilde{\Psi_\varphi}$  is a slice regular function defined on {a} non-open set and it fails to be extended slice-regularly to a larger slice-open set (or open set) in $\mathbb{H}$. Note that $\widetilde{\Omega_\varphi}$ is open in $\tau_\sigma$.
\end{exa}
\begin{exa}\label{exa314}
  Using the notations of the previous example,
 we now construct an example of a set $\widetilde{\Omega_\varphi}$ of the form given in \eqref{Omegatilde}, such that $\widetilde{\Omega_\varphi}$ is not in $\tau_s$ while it is in $\tau_\sigma$. We take a new curve $\gamma_s$, $s\in(0,1]$ which consist of three segments. We define
\begin{equation*}
	\gamma_s(t)=\begin{cases}
		\dfrac{1-(3-s)t}{2} i,& t\in[0,\frac{1}{3}],\ (\mbox{a line segment from $\frac{i}{2}$ to $\frac{si}{6}$}),
		\\3\left[\dfrac{2}{3}-t\right]\dfrac{si}{6}+3[t-\frac{1}{3}](i-1),& t\in(\frac{1}{3},\frac{2}{3}],\ (\mbox{a line segment from $\frac{si}{6}$ to $i-1$}),
		\\\dfrac{i}{1-t}-2i-1,& t\in(\frac{2}{3},1),\ (\mbox{a ray from $i-1$}).
	\end{cases}
\end{equation*}
Consider $\varphi:\mathbb{S}\rightarrow[0,1]$, defined by
\begin{equation*}
	\varphi(I)=\begin{cases}
		1,& I=\pm J,
		\\\frac{|I-J|}{2},& \mbox{otherwise}.
	\end{cases}
\end{equation*}
With this choice of $\varphi$, it is possible to prove that
\begin{equation*}
	\sigma\left(0,\mathbb{H}\backslash\widetilde{\Omega_\varphi}\right)\le\lim_{K\rightarrow I,K\neq\pm I} \mbox{dist}_{\mathbb{C}_K}(0,\mathcal{P}_K\circ\gamma_{\varphi(K)}[0,1))=0,
\end{equation*}
so that
 the point $0$ is not an interior point in $\widetilde{\Omega_\varphi}$ in the topology $\tau_\sigma$. The function $\widetilde{\Psi_\varphi}$ defined above is a slice regular function defined in a non-open set and it could not be extended slice regularly to a larger slice-open set (or open set or open set in $\tau_\sigma$) in $\mathbb{H}$.

Finally, we point out that by modifying the functions $\gamma_s$ and $\Psi$, one can construct many similar functions which are defined on sets in $\tau_s\backslash\tau_\sigma$.
\end{exa}

\section{Slice topology in cones on real alternative algebras}\label{sc-srla}

In \cite{Ghiloni2011001}, Ghiloni and Perotti introduced slice regular functions with values in a real alternative algebra, finite dimensional and with a fixed anti-involution, using stem functions. Stem functions were used in the literature also in relation with the Fueter mapping theorem, see \cite{DentoniSce,Fueter1934001,Rinehart1960001, Sce}. The slice regular functions in \cite{Ghiloni2011001} coincide with the class of slice regular functions over the quaternions \cite{Gentili2006001} and with the class of slice monogenic functions with values in a Clifford algebra, see \cite{Colombo2009002}, on some special open sets.

In this section, following \cite{Dou2020004}, we introduce a class of slice regular functions on a finite-dimensional real alternative algebra $A$ following the original idea of Gentili and Struppa, i.e. following Definition \ref{slreg}. Thus the functions slice regular in this sense do not coincide with those ones studied in \cite{Ghiloni2011001} and for this reason we sometimes called them weak slice regular and those ones in \cite{Ghiloni2011001} strong slice regular.

We recall that an algebra over the real numbers is said to be alternative if for any pair of elements  $x,y$ in the algebra
$$
x(xy)=x^2y,\qquad (xy)y=xy^2.
$$
It is immediate that every associative algebra is alternative. The converse does not hold, however every alternative algebra is power associative.

A real algebra $A$ is called left alternative, if for any $x,y\in A$,
\begin{equation*}
	x(xy)=(xx)y.
\end{equation*}
(See  \cite{Albert1949001}  for the notion of right alternative algebra, which can obviously adapted to our case).

 From now on we shall assume the following:
\begin{ass}\label{as-wa}
	Assume that $A\neq\{0\}$ is a finite-dimensional real unital left alternative algebra with $\mathbb{S}_A\neq\varnothing$, where
\begin{equation*}
	\mathbb{S}_A:=\{a\in A:a^2=-1\}.
\end{equation*}
\end{ass}
We  then define a map $L:A\rightarrow \textrm{End}_\mathbb{R}(A)$ by
\begin{equation}\label{La}
	L:a\mapsto L_a,
\end{equation}
where $L_a:A\rightarrow A$ is the right linear map given by the left multiplication by $a$
\begin{equation*}
	L_a:x\mapsto ax.
\end{equation*}

Note that for any $I\in\mathbb{S}_A$, we have
\begin{equation*}
	(L_I)^2(x)=I(Ix)=-x,
\end{equation*}
thus $L_I$ is a complex structure on $A$.

Let $I,J\in\mathbb{S}_A$ with $I\neq\pm J$. It is easy to check that
\begin{equation*}
	\mathbb{C}_I\cap\mathbb{C}_J=\mathbb{R}.
\end{equation*}

\subsection{Slice topology}

In this subsection, we discuss the slice topology whose definition, originally given over the quaternions, can  be extended to a suitable set $Q_A$ that we call the quadratic cone in the real alternative algebra $A$:
\begin{equation*}
	Q_A:=\bigcup_{I\in\mathbb{S}_A}\mathbb{C}_I.
\end{equation*}

\begin{lem} The set
	$$\tau_s(Q_A):=\{\Omega\subset Q_A:\Omega\ \mbox{is slice-open}\}$$ is a topology of $Q_A$.
\end{lem}

\begin{defn}
	We call $\tau_s(Q_A)$ the slice topology on $Q_A$. Open sets, connected sets and paths in the slice topology are called  slice-open sets, { slice-connected sets} and slice-paths.
\end{defn}

Since $A$ is finite-dimensional real vector space of dimension, say, there is a Euclidean topology on $A$ identified with Euclidean space $\mathbb R^m$. Since $A$ admits a complex structure, $m$ is an even number and so $m=2n$. The cone $Q_A$, as a subset of $A$, also has a Euclidean topology (i.e. the subspace topology induced by $A$).

As in the quaternionic case, we do not use the terminology slice-domain to denote a domain in the slice topology, and we will use instead the term {\em slice topology-domain}, in short, {\em st-domain}.

\begin{defn}
	A set  $\Omega$ in $Q_A$ is called classical slice domain, in short s-domain,   if  $\Omega$ is a domain in $Q_A$ in the Euclidean topology,
	$$\Omega_\mathbb{R}:=\Omega\cap \mathbb R\neq\varnothing, $$ and $\Omega_I$  is a domain in $\mathbb{C}_I$ for any $I\in\mathbb{S}_A$.
\end{defn}

The slice topology on $A$ has similar properties of the slice topology in $\mathbb{H}$ and in particular:

\begin{prop}
	$(Q_A,\tau_s)$ is a Hausdorff space and $\tau\subset\tau_s$.
\end{prop}

We remark that the slice topology $\tau_s$ is not always strictly finer  than the Euclidean topology $\tau$ as the following simple example shows.

\begin{exa}
	$\mathbb{C}$ is an algebra satisfing Assumption \ref{as-wa}. The slice topology and Euclidean topology coincide on $\mathbb{C}$.
\end{exa}
The terminology and the results in Section 3 can be stated in this more general setting, so we do not repeat them and we refer the reader to \cite{Dou2020004}. We only mention the following results:

\begin{prop}
	The topological space $(Q_A,\tau_s)$ is connected, local path-connected and path-connected.
\end{prop}

The notion of st-domain is also a generalization of the notion of s-domain.

\begin{cor}
	A set $\Omega\subset Q_A$ is an st-domain if $\Omega_\mathbb{R}\neq\varnothing$ and $\Omega_I$ is a domain in $\mathbb{C}_I$ for each $I\in\mathbb{S}_A$.
\end{cor}

\subsection{Main results in left alternative algebra}
The slice regularity given in Definition \ref{slreg} can be extended to the case of a left alternative algebra satisfying Assumption \ref{as-wa}, and in this subsection we state some main results about this class of functions.
\begin{defn}\label{slregA}
	Let $\Omega\in\tau_s(Q_A)$. A function $f:\Omega\rightarrow A$ is called slice regular if and only if for each $I\in\mathbb{S}_A$, $f_I:=f|_{\Omega_I}$ is (left $I$-)holomorphic, i.e. $f_I$ is real differentiable and
	\begin{equation*}
		\frac{1}{2}\left(\frac{\partial}{\partial x}+I\frac{\partial}{\partial y}\right)f_I(x+yI)=0,\qquad\mbox{on}\qquad\Omega_I.
	\end{equation*}
\end{defn}
Let $A$ be a left alternative algebra and $I\in\mathbb{S}_A$. The set $\{\theta_1,...,\theta_n\}\subset A$ is called an $I$-basis, if
\begin{equation*}
	\{\theta_1,I(\theta_1),\theta_2,I(\theta_2),...,\theta_n,I(\theta_n)\}
\end{equation*}
is a real basis of $A$. Since $L_J$ (see \eqref{La}) is a complex structure on $A$, there is {a} $J$-basis for all $J\in\mathbb{S}_A$.
\smallskip

Various results that we have stated in the quaternionic case, e.g. the Splitting Lemma and the Identity Principle, hold also in this more general case.

In the sequel, we need the following notations and definitions: for any $I\in\mathbb{S}_A$ we can define an isomorphism $\mathcal{P}_I:\ \mathbb{C}\longrightarrow \mathbb{C}_I$ such that {$\mathcal P_I(x+yi)= x+yI$}.
For each path $\gamma:[0,1]\rightarrow\mathbb{C}$, we define its corresponding path in $\mathbb{C}_I^d$ by
\begin{equation*}
	\gamma^{I}:=\mathcal{P}^I\circ\gamma.
\end{equation*}
Finally, for any $\Omega\subset Q_{A}$ and $\gamma\in\mathscr{P}(\mathbb{C})$ we set
\begin{equation*}
	\mathscr{P}(\mathbb{C},\Omega):=\{\delta\in\mathscr{P}(\mathbb{C}):\exists\ I\in\mathbb{S}_A,\mbox{ s.t. }\delta^{I}\subset\Omega\}
\end{equation*}
and
\begin{equation*}
	\mathbb{S}_A(\gamma,\Omega):=\{I\in\mathbb{S}_A:\gamma^{I}\subset\Omega\}.
\end{equation*}
Mimicking Definition  \ref{Def36}, given $\Omega\subset Q_{A}$, we say that a function $f:\Omega\rightarrow A$ is path-slice if there is a function $F:\mathscr{P}(\mathbb{C},\Omega)\rightarrow A^{2\times 1}$ such that
	\begin{equation*}
		f\circ\gamma^{I}(1)=(1,L_I)F(\gamma),
	\end{equation*}
	for each $\gamma\in\mathscr{P}(\mathbb{C},\Omega)$ and $I\in\mathbb{S}_A(\gamma,\Omega)$.
	
We have the following results:
\begin{thm}
	Let $\Omega\in\tau_s(Q_{A})$ and $f:\Omega\rightarrow A$ be slice regular. Then $f$ is path-slice.
\end{thm}

Let $J,K\in\mathbb{S}_A$, we have
\begin{equation*}
	L_J(L_J-L_K)=-1-L_JL_K=L_KL_K-L_JL_K=(L_K-L_J)L_K.
\end{equation*}
If $L_J-L_K$ is invertible, then
\begin{equation*}
	(L_J-L_K)^{-1}L_J=-L_K(L_J-L_K)^{-1}.
\end{equation*}

One can easily verify that
\begin{equation}\label{L-1}
	\left(\begin{matrix}
		1&L_J\\1&L_K
	\end{matrix}\right)^{-1}
	=\left(\begin{matrix}
		(L_J-L_K)^{-1}L_J&-(L_J-L_K)^{-1}L_K\\(L_J-L_K)^{-1}&-(L_J-L_K)^{-1}
	\end{matrix}\right).
\end{equation}

\begin{prop}\label{pr-cp}
	(Standard path-representation Formula) Let $\Omega$ be a slice-open set in $\mathcal{W}_A^1$, $f:\Omega\rightarrow A$ be a weak slice regular function, $\gamma\in\mathscr{P}(\mathbb{C}^1,\Omega)$ and $I,J_1,J_2\in \mathbb{S}_A(\gamma,\Omega)$ with $L_{J_1}-L_{J_2}$ invertible. Then
	\begin{equation}\label{eq-fg}
		f\circ\gamma^I=(1,L_I)
		\begin{pmatrix}
			1&L_{J_1}\\1&L_{J_2}
		\end{pmatrix}^{-1}
		\begin{pmatrix}
			f\circ\gamma^{J_1}\\f\circ\gamma^{J_2}
		\end{pmatrix},
	\end{equation}
	where $\begin{pmatrix}
		1&L_{J_1}\\1&L_{J_2}
	\end{pmatrix}^{-1}$ satisfies \eqref{L-1} and $\mathbb{S}_A(\gamma,\Omega):=\{I\in\mathbb{S}_A:Ran(\gamma^I)\subset\Omega\}$.
\end{prop}

\subsection{Slice monogenic functions}
We now consider the special case in which $A=\mathbb{R}_n$, the real Clifford algebra over $n$ imaginary units $\mathbf{e}_1,...,\mathbf{e}_n$ satisfying $\mathbf{e}_\imath \mathbf{e}_\jmath+\mathbf{e}_\jmath \mathbf{e}_\imath=-2\delta_{\imath,\jmath}$, where
\begin{equation*}
	\delta_{\imath\jmath}={\begin{cases}0,&\imath\neq\jmath,\\-1,&\imath=\jmath.\end{cases}}
\end{equation*}
We recall that the slice regular functions over a real Clifford algebra are also called slice monogenic functions, and they were firstly introduced in \cite{Colombo2009002}.

\begin{exa}
In \cite{Colombo2009002}, the class of slice monogenic functions is defined on
\begin{equation*}
	\mathbf{R}^{n+1}:=\bigcup_{ I\in\mathbb{S}_{\mathbb{R}_n}^\circ}\mathbb{C}_I,
\end{equation*}
where
\begin{equation*}
	\mathbb{S}_{\mathbb{R}_n}^\circ:=\{x_1\mathbf{e}_1+\cdots x_n \mathbf{e}_n:x_1^2+\cdots+x_n^2=1\}.
\end{equation*}
The cone $\mathbf{R}^{n+1}$ is a $n+1$-dimensional real vector space, i.e.
\begin{equation*}
	\mathbf{R}^{n+1}=\mathbb{R}\langle1,\mathbf{e}_1,...,\mathbf{e}_n\rangle\cong\mathbb{R}^{n+1}
\end{equation*}
However, as observed in \cite{Ghiloni2011001}, slice monogenic functions can be defined on a larger set, namely the quadratic cone of $\mathbb{R}_n$:
\begin{equation*}
	Q_{\mathbb{R}_n}:=\bigcup_{I\in\mathbb{S}_{\mathbb{R}_n}}\mathbb{C}_I
\end{equation*}
where
\begin{equation*}
	\mathbb{S}_{\mathbb{R}_n}:=\{I\in\mathbb{R}_n:I^2=-1\}.
\end{equation*}
The slice monogenic functions in \cite{Ghiloni2011001} are defined on symmetric open sets in $Q_{\mathbb{R}_n}$. Since a Clifford algebra is a special case of an alternative algebra, we can consider the more general case of slice monogenic functions defined on a slice-open set in $Q_{\mathbb{R}_n}$ where the definition of slice monogenic is given by adapting Definition \ref{slregA} to the present case. Thus all the results in Section 4 are valid in this case.
\end{exa}

\section{Weak slice regular functions over slice-cones}

Some results in \cite{Dou2020004} are given in a greater generality:  in fact the weak slice regular functions can be considered in the case of functions which are $\mathbb{R}^{2n}$-valued and defined on open sets in the topological space $(\mathcal{W}_\mathcal{C}^d,\tau_s)$, where $\mathcal{W}_\mathcal{C}^d$ is a suitable weak slice cone in $[\End(\mathbb{R}^{2n})]^d$.  In \cite{Dou2020004} we proved various results for these functions, among which a representation formula, and we recall some of them in this section.

We denote by $\mathfrak{C}_n$ the set of complex structures on $\mathbb{R}^{2n}$, i.e.
\begin{equation*}
	\mathfrak{C}_n:=\left\{T\in\End\left(\mathbb{R}^{2n}\right)\ :\ T^2=-1\right\},
\end{equation*}
where the identity map $id_{\mathbb{R}^{2n}}$ on $\mathbb{R}^{2n}$ is denoted by $1$.

Let $\mathcal{C}\subset\mathfrak{C}_n$ {  with $\mathcal{C}=-\mathcal{C}$}. We call
\begin{equation*}
	\mathcal{W}_\mathcal{C}^d:=\bigcup_{I\in\mathcal{C}}\mathbb{C}_I^d
\end{equation*}
the $d$-dimensional weak slice-cone of $\mathcal{C}$, where
\begin{equation*}
	\mathbb{C}_I^d:=\left(\mathbb{R}+\mathbb{R}I\right)^d.
\end{equation*}

The slice topology on $\mathcal{W}_{\mathcal{C}}^d$ is defined by
\begin{equation*}
	\tau_s\left(\mathcal{W}_{\mathcal{C}}^d\right):=\left\{\Omega\subset \mathcal{W}_{\mathcal{C}}^d\ :\ \Omega_I\in \tau(\mathbb{C}_I^d),\ \forall\ I\in\mathcal{C}\right\},
\end{equation*}
where
\begin{equation*}
	\Omega_I:=\Omega\cap\mathbb{C}_I^d.
\end{equation*}

{\bf Convention}: Let $\Omega\subset\mathcal{W}_{\mathcal{C}}^d$. Denote by $\tau_s(\Omega)$ the subspace topology induced by $\tau_s\left(\mathcal{W}_{\mathcal{C}}^d\right)$. Open sets, domains, connected sets and paths in $\tau_s(\Omega)$ are called slice-open sets, slice-domains, slice-connected sets and slice-paths in $\Omega$, respectively. We write $x+yI\in\Omega$ short for $x+yI\in\Omega$ with { $x,y\in\mathbb{R}^d$} and $I\in\mathcal{C}$.

\begin{defn}
	Let $\Omega\in\tau_s(\mathcal{W}_\mathcal{C}^d)$. A function $f:\Omega\rightarrow\mathbb{R}^{2n}$ is called weak slice regular if and only if for each $I\in\mathcal{C}$, $f_I:=f|_{\Omega_I}$ is (left $I$-)holomorphic, i.e. $f_I$ is real differentiable and for each $\ell=1,2,...,d$,
	\begin{equation*}
		\frac{1}{2}\left(\frac{\partial}{\partial x_\ell}+I\frac{\partial}{\partial y_\ell}\right)f_I(x+yI)=0,\qquad\mbox{on}\qquad\Omega_I.
	\end{equation*}
\end{defn}

For any $I\in\mathcal{C}$, the set $\{\xi_1,...,\xi_n\}\subset\mathbb{R}^{2n}$ is called an $I$-basis of $\mathbb{R}^{2n}$ if
\begin{equation*}
	\{\xi_1,...,\xi_n,I(\xi_1),...,I(\xi_n)\}
\end{equation*}
is a basis of $\mathbb{R}^{2n}$ as a real vector space.

\begin{lem}\label{lm-wsp1}
	(Splitting Lemma) Let $\Omega\in\tau_s\left(\mathcal{W}_\mathcal{C}^d\right)$. A function $f:\Omega\rightarrow\mathbb{R}^{2n}$ is weak slice regular if and only if for any $I\in\mathcal{C}$ and $I$-basis $\{\xi_1,....,\xi_n\}$, there are $n$ holomorphic functions $F_1,...,F_n:\Omega_I\subset\mathbb{C}_I^d\rightarrow\mathbb{C}_I$, such that
	\begin{equation*}
		f_I=\sum_{\ell=1}^n(F_\ell\xi_\ell).
	\end{equation*}
\end{lem}

\begin{thm}\label{tm-ip}
	(Identity Principle) Let $\Omega$ be a slice-domain in $\mathcal{W}_\mathcal{C}^d$ and $f,g:\Omega\rightarrow\mathbb{R}^{2n}$ be weak slice regular. Then the following statements holds.
	\begin{enumerate}[\upshape (i)]
		\item\label{it-ip1} If $f=g$ on a non-empty open subset $U$ of $\Omega_{\mathbb{R}}$, then $f=g$ on $\Omega$.
		\item\label{it-ip2} If $f=g$ on a non-empty open subset $U$ of $\Omega_I$ for some $I\in\mathcal{C}$, then $f=g$ on $\Omega$.
	\end{enumerate}
\end{thm}

For any $I\in\mathcal{C}$, let us choose a fixed $I$-basis of $\mathbb{R}^{2n}$ which is denoted by
\begin{equation}\label{eq-ti}
	\theta^I:=\left\{\theta^I_1,...,\theta^I_n\right\},
\end{equation}
and let us consider the ${2^n\times 2^n}$ real matrix $D_I$ given by
\begin{equation}\label{eq-dib} D_I:=\begin{pmatrix}\theta_1^I&\cdots&\theta_n^I&I\theta_1^I&\cdots&I\theta_n^I\end{pmatrix}.
\end{equation}

By \cite[Proposition 2.3]{Dou2020004}, we can define the Moore-Penrose inverse of $J\in\End\left(\mathbb{R}^{2n}\right)^{k\times\ell}$.

\begin{defn}
	For each $J\in\End\left(\mathbb{R}^{2n}\right)^{k\times\ell}$, define by $J^+$ the unique matrix in $\End\left(\mathbb{R}^{2n}\right)^{\ell\times k}$ that satisfies the Moore-Penrose conditions:
	\begin{enumerate}[\upshape (i)]
		\item\label{it-jmp1} $JJ^+J=J$.
		\item $J^+JJ^+=J^+$.
		\item $(JJ^+)^*=JJ^+$.
		\item $(J^+J)^*=J^+J$.

	\end{enumerate}
	Here $J^*$ is the adjoint matrix of $J$, i.e. the unique matrix such that
	\begin{equation*}
		\langle x,Jy\rangle=\langle J^* x,y\rangle,\qquad\forall\ x\in\left(\mathbb{R}^{2n}\right)^{k},\ y\in\left(\mathbb{R}^{2n}\right)^{\ell}.
	\end{equation*}
	We call $J^+$ the Moore-Penrose inverse of $J$.
\end{defn}

Let $J=(J_1,...,J_k)^T\in\mathcal{C}^k$. We denote
\begin{equation*}
	D_J:=\begin{pmatrix}
		D_{J_1}\\&\ddots\\&&D_{J_k}
	\end{pmatrix}\qquad\mbox{and}\qquad\zeta(J):=\begin{pmatrix}
	1&J_1\\\vdots &\vdots\\1&J_k
\end{pmatrix},
\end{equation*}
where $D_{J_\ell}$ is defined by \eqref{eq-dib}.

We set
\begin{equation*}
	\zeta^+(J):=[\Im^{-1}_{D_J}\cdot\zeta(J)]^+ \Im^{-1}_{D_J}
\end{equation*}
and we call it the $J$-slice inverse of $\zeta(J)$, where $\Im^{-1}_{D_J}:=\Im^{-1}(D_J)$ is the unique $\End\left(\mathbb{R}^{2n}\right)$-valued $k\times k$ matrix such that
\begin{equation*}
	\Im^{-1}_{D_J}(a)=D_J a,\qquad\forall\ a\in\left(\mathbb{R}^{2n}\right)^{k}=\mathbb{R}^{2nk}.
\end{equation*}

Let $\gamma:[0,1]\rightarrow\mathbb{C}^d$ and $I\in\mathcal{C}$. We define the image path of $\gamma$ in $\mathbb{C}_I^d$ by
\begin{equation*}
	\gamma^{I}:=\Psi_i^{I}\circ\gamma,
\end{equation*}
where $\Psi_i^I:\mathbb{C}^d\rightarrow\mathbb{C}_I^d$, $x+yi\mapsto x+yI$
is an isomorphism.
As we already did in the preceding sections, we define:
\begin{equation*}
	\mathscr{P}(\mathbb{C}^d):=\{\gamma:[0,1]\rightarrow\mathbb{C}^d,\ \gamma\ \mbox{is a path s.t. }\gamma(0)\in\mathbb{R}^d\};
\end{equation*}
for any $\Omega\subset\mathcal{W}_\mathcal{C}^d$ and for an arbitrary, but fixed $\gamma\in\mathscr{P}\left(\mathbb{C}^d\right)$ we define
\begin{equation*} \mathscr{P}\left(\mathbb{C}^d,\Omega\right):=\left\{\delta\in\mathscr{P}\left(\mathbb{C}^d\right):\exists\ I\in\mathcal{C},\mbox{ s.t. } Ran(\delta^{I})\subset\Omega\right\};
\end{equation*}
\begin{equation*}
	{\mathcal{C}(\Omega,\gamma)}:=\left\{I\in\mathcal{C}: Ran(\gamma^{I})\subset\Omega\right\},
\end{equation*}
where $Ran(\cdot)$ denotes the range.

Let $J=(J_1,...,J_k)^T\in\mathcal{C}^k$, $\Omega\subset\mathcal{W}_\mathcal{C}^d$ and $\gamma\in\mathscr{P}(\mathbb{C}^d,\Omega)$. We define
\begin{equation}\label{eq-ck} \mathcal{C}_{ker}(J):=\left\{I\in\mathcal{C}:\ker(1,I)\supset\ker[\zeta(J)]=\bigcap_{\ell=1}^k\ker(1,J_\ell)\right\},
\end{equation}
and
\begin{equation*}
	\mathcal{C}(\Omega,\gamma,J):=\mathcal{C}(\Omega,\gamma)\cap\mathcal{C}_{ker}(J),
\end{equation*}
where $\ker(\cdot)$ is the kernel of a map, and $\ker(1,J_\ell)$ stands for $\ker((1,J_\ell))$.

The results below were originally proved in \cite{Dou2020004}, Section 5:
\begin{lem}\label{lm-los}
	Let $\Omega\subset\mathcal{W}_\mathcal{C}^d$, $\gamma\in\mathscr{P}\left(\mathbb{C}^d,\Omega\right)$ and $J=(J_1,...,J_k)^T\in\left[\mathcal{C}(\Omega,\gamma)\right]^k$. Then there is a domain $U$ in $\mathbb{C}^d$ containing $\gamma([0,1])$ such that
	\begin{equation*}
		\Psi_i^{J_\ell}(U)\subset\Omega,\qquad\ell=1,...,k.
	\end{equation*}
\end{lem}

\begin{lem}\label{lm-lut}(Extension Lemma)
		Let $U\in\tau(\mathbb{C}^d)$, $I\in\mathcal{C}$ and $J=(J_1,...,J_k)^T\in\mathcal{C}^k$. If $g_\ell:\Psi_i^{J_\ell}(U)\rightarrow\mathbb{R}^{2n}$, $\ell=1,...,k$ are holomorphic, then the function $g[I]:\Psi_i^I(U)\rightarrow\mathbb{R}^{2n}$ defined by
		\begin{equation*}
			g[I](x+yI)=(1,I)
			\zeta^+(J)
			g(x+yJ),\qquad\forall\ x+yi\in U,
		\end{equation*}
		where
		\begin{equation*}
			g(x+yJ)=\begin{pmatrix}
				g_1(x+yJ_1)\\\vdots\\g_k(x+yJ_k)
			\end{pmatrix}
		\end{equation*}
		is holomorphic.
		
		Moreover, if $U_\mathbb{R}:=U\cap\mathbb{R}^d\neq\varnothing$, $g_1=\cdots=g_k$ on $U_\mathbb{R}$ and $I\in\mathcal{C}_{ker}(J)$, then
		\begin{equation}\label{eq-gg}
			g[I]=g_1=\cdots=g_k\qquad\mbox{on}\qquad U_\mathbb{R}.
		\end{equation}
\end{lem}
We now present the so-called path-representation formula, see Theorem 6.1 in \cite{Dou2020004}, which is a crucial result in this function theory.
\begin{thm}\label{th-prf}
	(Path-representation Formula) Let $\Omega$ be a slice-open set in $\mathcal{W}_\mathcal{C}^d$, $\gamma\in\mathscr{P}(\mathbb{C}^d,\Omega)$, $J=(J_1,J_2,...,J_k)^T\in\left[\mathcal{C}(\Omega,\gamma)\right]^k$ and $I\in\mathcal{C}(\Omega,\gamma,J)$.
	If $f:\Omega\rightarrow\mathbb{R}^{2n}$ is weak slice regular, then
	\begin{equation*}
		f\circ\gamma^I=(1,I)\zeta^+(J)(f\circ\gamma^J),
	\end{equation*}
	where
	\begin{equation}\label{eq-fgj}
		f\circ\gamma^J:=\begin{pmatrix}
			f\circ\gamma^{J_1}\\\vdots\\f\circ\gamma^{J_k}
		\end{pmatrix}.
	\end{equation}
\end{thm}

\begin{defn}
	Let $\mathcal{C}'\subset\mathcal{C}$ and $J\in(\mathcal{C}')^k$. We say that $J$ is a slice-solution of $\mathcal{C}'$ if
	\begin{equation*}
		\mathcal{C}'\subseteq\mathcal{C}_{ker} (J).
	\end{equation*}
\end{defn}

\begin{exa}
	Let $\Omega\subset\mathcal{W}_\mathcal{C}^d$, $\gamma\in\mathscr{P}(\mathbb{C}^d,\Omega)$ and $J\in \left[\mathcal{C}(\Omega,\gamma)\right]^k$. $J$ is a slice-solution of $\mathcal{C}(\Omega,\gamma)$ if and only if
	\begin{equation*}
		\mathcal{C}(\Omega,\gamma)=\mathcal{C}(\Omega,\gamma,J).
	\end{equation*}
\end{exa}

\begin{remark}\label{rm-lj}
	Note that for any { $I\in\mathcal{C}$,}
		\begin{equation*}
			\begin{pmatrix}
				1&I\\1&-I
			\end{pmatrix}^{-1}={\frac{1}{2}}\begin{pmatrix}
			1&1\\-I&I
		\end{pmatrix}
		\end{equation*}
	and
	\begin{equation*}
		ker\left(\zeta\left((I,-I)^T\right)\right)=ker\begin{pmatrix}
			1&I\\1&-I
		\end{pmatrix}=0.
	\end{equation*}
	It is easy to check that, by definition (see \eqref{eq-ck}), $J\in\mathcal{C}^k$ is a slice-solution of $\mathcal{C}$ if and only if
	\begin{equation*}
		\bigcap_{\ell=1}^k\ker(1,J_\ell)=\{0\}.
	\end{equation*}
	In another words, if $J$ is not a slice-solution, we can take a non-zero element
	\begin{equation*}
		a=\begin{pmatrix}
			a_1\\a_2
		\end{pmatrix}\in\bigcap_{\ell=1}^k\ker(1,J_\ell)\subset\left(\mathbb{R}^{2n}\right)^2.
	\end{equation*}
	Then for each $K\in\mathcal{C}_{ker}(J)$, we have $(1,K)a=0$.
\end{remark}

\begin{prop}
	Let $\mathcal{C}'\subset\mathcal{C}$. Then there is at least one slice-solution $J\in(\mathcal{C}')^k$ of $\mathcal{C}'$ for some $k\in\mathbb{N}_+$.
\end{prop}

\begin{proof}
	This proposition follows the reasoning in \cite[Proposition 6.4]{Dou2020004}.
\end{proof}
As a consequence we deduce the following:
\begin{cor}\label{co-lo}
	Let $\Omega\in\tau_s(\mathcal{W}_\mathcal{C}^d)$, $\gamma\in\mathscr{P}(\mathbb{C}^d,\Omega)$, and $J\in\left[\mathcal{C}(\Omega,\gamma)\right]^k$ be a slice-solution of $\mathcal{C}(\Omega,\gamma)$. If $f:\Omega\rightarrow\mathbb{R}^{2n}$ is weak slice regular, then
	\begin{equation*}
		f\circ\gamma^I=(1,I)\zeta^+(J)(f\circ\gamma^J),\qquad\forall\ I\in\mathcal{C}(\Omega,\gamma),
	\end{equation*}
	where $f\circ\gamma^J$ is defined by \eqref{eq-fgj}.
\end{cor}

Next definition gives the terminology to mention a result, originally proved in
 \cite[Corollary 6.9]{Dou2020003}, which shows that also in this more general setting a slice regular function is path-slice.
\begin{defn}\label{df-ps}
	A function $f:\Omega\rightarrow\mathbb{R}^{2n}$ with $\Omega\subset\mathcal{W}_\mathcal{C}^d$ is called path-slice, if for any $\gamma\in\mathscr{P}\left(\mathbb{C}^d,\Omega\right)$, there is a function $F_\gamma:[0,1]\rightarrow\left(\mathbb{R}^{2n}\right)^{2\times 1}$ such that
	\begin{equation*}
		f\circ\gamma^I=(1,I)F_\gamma,\qquad\forall\ I\in\mathcal{C}(\Omega,\gamma).
	\end{equation*}
	
	We call $\{F_\gamma\}_{\gamma\in\mathscr{P}(\mathbb{C})}$   a (path-)stem system of the path-slice function $f$.
\end{defn}

\begin{prop}
	Each slice regular function $f:\Omega\rightarrow\mathbb{R}^{2n}$ with $\Omega\in\tau_s\left(\mathcal{W}_\mathcal{C}^d\right)$ is path-slice.
\end{prop}

\begin{prop}\label{pr-jjk}
	Let $J=(J_1,...,J_k)^T\in\mathcal{C}^k$. Then
	\begin{equation}\label{eq-il}
		Ran\left[id_{(\mathbb{R}^{2n})^{2\times 1}}-\zeta^+(J)\zeta(J)\right]=\ker[\zeta(J)]=\bigcap_{\ell=1}^k \ker(1,J_\ell).
	\end{equation}
\end{prop}

\section{Extension Theorem}

In this section, we give an extension theorem for slice regular functions. This result provides a tool for extending slice regular functions to larger definition domains. We also prove a general Path-representation Formula.

\begin{cor}
	(General path-representation Formula) Let $\Omega\subset\mathcal{W}_\mathcal{C}^d$, $\gamma\in\mathscr{P}(\mathbb{C}^d,\Omega)$, and $J=(J_1,...,J_k)^T\in\left[\mathcal{C}(\Omega,\gamma)\right]^k$. If $f:\Omega\rightarrow\mathbb{R}^{2n}$ is path-slice, then for each $I\in\mathcal{C}(\Omega,\gamma)$ and $t\in[0,1]$,
	\begin{equation}\label{eq-fci}
		f\circ\gamma^I(t)\in(1,I)\left[\zeta^+(J)(f\circ\gamma^J(t))+Ran(id_{(\mathbb{R}^{2n})^{2\times 1}}-\zeta^+(J)\zeta(J))\right],
	\end{equation}
	where $f\circ\gamma^J$ is defined by \eqref{eq-fgj}.
\end{cor}

\begin{proof}
	Let $K\in\left[\mathcal{C}(\Omega,\gamma)\right]^m$ be a slice-solution of $\mathcal{C}(\Omega,\gamma)$. Fix $I\in\mathcal{C}(\Omega,\gamma)$. According to Corollary \ref{co-lo},
	\begin{equation}\label{eq-fcz}
		f\circ\gamma^I=(1,I)\zeta^+(K)(f\circ\gamma^K),
	\end{equation}
	and
	\begin{equation}\label{eq-fcg}
		f\circ\gamma^J=\zeta(J)\zeta^+(K)(f\circ\gamma^K).
	\end{equation}
	Note that $J_1,...,J_k\in\mathcal{C}(\Omega,\gamma,J)$. By Path-representation Formula \ref{th-prf},
	\begin{equation}\label{eq-fcj}
		f\circ\gamma^J=\zeta(J)\zeta^+(J)(f\circ\gamma^J).
	\end{equation}
	By \eqref{eq-fcg}$-$\eqref{eq-fcj}, we have
	\begin{equation}\label{eq-zkf}
		\zeta^+(K)(f\circ\gamma^K)-\zeta^+(J)(f\circ\gamma^J)\subset\ker[\zeta(J)]= Ran\left[id_{(\mathbb{R}^{2n})^{2\times 1}}-\zeta^+(J)\zeta(J)\right].
	\end{equation}
	It is clear that \eqref{eq-fci} holds by \eqref{eq-fcz} and \eqref{eq-zkf}.
\end{proof}

For each $\Omega\subset\mathcal{W}_{\mathcal{C}}^d$, we define
\begin{equation*}
	\Omega_s:=\{x+yi\in\mathbb{C}^d:x+yI\in\Omega\}.
\end{equation*}

\begin{defn}
	A function $f:\Omega\rightarrow\mathbb{R}^{2n}$ with $\Omega\subset\mathcal{W}_\mathcal{C}^d$ is called slice, if there is a function $F:\Omega_s\rightarrow\left(\mathbb{R}^{2n}\right)^{2\times 1}$ such that
	\begin{equation*}
		f(x+yI)=(1,I)F(x+yi),\qquad\forall\ x+yI\in\Omega.
	\end{equation*}
	
	We call $F$ a stem function of the slice function $f$.
\end{defn}

\begin{defn}
	A set $\Omega\subset\mathcal{W}_\mathcal{C}^d$ is called axially symmetric, if for each $x+yI\in\Omega$, then $x+y\mathcal{C}\subset\Omega$.
\end{defn}

\begin{prop}
	Let $\Omega$ be an axially symmetric slice-domain in $\mathcal{W}^d_\mathcal{C}$. Then $\Omega_I$ is a domain in $\mathbb{C}_I^d$ for each $I\in\mathcal{C}$.
	Moreover, if $\Omega$ is non-empty, then $\Omega_\mathbb{R}\neq\varnothing$.
\end{prop}

\begin{proof}
	The proof follows the reasoning in \cite[Remark 7.7]{Dou2020001}.
\end{proof}

\begin{prop}\label{pr-lo}
	Let $\Omega\subset\mathcal{W}_{\mathcal{C}}^d$ be an axially symmetric slice-domain and $f:\Omega\rightarrow\mathbb{R}^{2n}$ be path-slice (or slice regular). Then $f$ is a slice function.
\end{prop}

\begin{proof}
	Since slice regular function are path-slice, we can assume that $f$ is path-slice.
	
	For each $z=x+yi\in\Omega_s$. We choose a fixed complex structure $K_z\in\mathcal{C}$. By definition, the point $q:=x+yK_z$ belongs to $\Omega$. Since $\Omega_{K_z}$ is a domain in $\mathbb{C}_{K_z}^d$ and $\Omega_\mathbb{R}\neq\varnothing$. We can choose a fixed path $\gamma_{z}\in\mathscr{P}\left(\mathbb{C}^d,\Omega\right)$ such that $(\gamma_z)^{K_z}$ is a path in $\Omega_{K_z}\subset\mathbb{C}_{K_z}^d$ from a point in $\Omega_\mathbb{R}$ to $q$. As $\Omega$ is an axially symmetric slice-domain, it follows from definition that $\mathcal{C}(\Omega,\gamma_z)=\mathcal{C}$. Since $f$ is path-slice, there is a function $G_{\gamma_{x+yi}}:[0,1]\rightarrow(\mathbb{R}^{2n})^{2\times 1}$ such that
	\begin{equation*}
		f(\gamma_{x+yi}^I)=(1,I)G_{\gamma_{x+yi}},\qquad\forall\ I\in\mathcal{C}.
	\end{equation*}
 	It implies that
	\begin{equation*}
		f(x+yI)=f(\gamma_{x+yi}^I(1))=(1,I)G_{\gamma_{x+yi}}(1),\qquad\forall\ I\in\mathcal{C}.
	\end{equation*}
	It is clear that $f$ is a slice function with its stem function $F:\Omega_s\rightarrow(\mathbb{R}^{2n})^{2\times 1}$ defined by
	\begin{equation*}
		F(x+yi):=\zeta^+(J)G_{\gamma_{x+yi}}(1).
	\end{equation*}
\end{proof}

For any $J=(J_1,...,J_k)^T\in\mathcal{C}^k$, we set
\begin{equation*}
	\tau[J]:=\{(U_1,...,U_k)^T:U_\ell\in\tau\left(\mathbb{C}_{J_\ell}^d\right),\ell=1,...,k\}.
\end{equation*}

We define a set
\begin{equation*}
	\mathcal{C}_*^k:=\left\{(J_1,...,J_k)^T\in\mathcal{C}^k:J_\imath\neq\pm J_\jmath,\ \forall\ \imath\neq\jmath\right\}
\end{equation*}

\begin{defn}
	Let {$J=(J_1,...,J_k)^T\in\mathcal{C}_*^k$}. A subset $\mathcal{C}^+$ of $\mathcal{C}$ is called a slice-half subset of $\mathcal{C}$ with respect to $J$, if {$J_1,...,J_k\in\mathcal{C}^+$ and}
	\begin{equation*}
		\mathcal{C}=\left(-\mathcal{C}^+\right)\bigsqcup\mathcal{C}^+.
	\end{equation*}
\end{defn}
\noindent
{\bf Convention.} For each $J\in\mathcal{C}_*^k$. We fix a slice-half subset $\mathcal{C}_J^+$ of $\mathcal{C}$ with respect to $J$.

Associated with $J\in\mathcal{C}_*^k$ and $U\in\tau[J]$,
we introduce the following sets:
\begin{equation*}
	U_\mathbb{C}:=\{x+yi\in\mathbb{C}^d:x+yJ_\ell\in U_\ell,\ \ell=1,...,k\}
\end{equation*}
\begin{equation*}
	{U_\mathbb{C}^*}:=\{x+yi\in\mathbb{C}^d:x\pm yJ_\ell\in U_\ell,\ \ell=1,...,k\}
\end{equation*}
\begin{equation*}
	U_\Delta:=\bigcup_{x+yi\in U_\mathbb{C}} x+y\mathcal{C}_{ker}(J),\qquad U_\Delta^+:=\bigcup_{x+yi\in U_\mathbb{C}} x+y\mathcal{C}^+_J
\end{equation*}
and
\begin{equation*}
	U_\Delta^*:=\bigcup_{x+yi\in U_\mathbb{R}^*} x+y\mathcal{C},
\end{equation*}
where {$U_\mathbb{R}^*$ is the union of the connected components of $U_\mathbb{C}^*$} intersecting $\mathbb{R}^d$.
Finally, we define
\begin{equation*}
	U_\Delta^\backsim:=\begin{cases}
		U_\Delta^+\bigcup\left(\bigcup_{\ell=1}^k U_\ell\backslash\mathbb{R}^d\right),\qquad&\mbox{if $J$ is a slice solution,}
		\\U_\Delta\bigcup U_\Delta^* \bigcup\left(\bigcup_{\ell=1}^k U_\ell\backslash\mathbb{R}^d\right),\qquad&\mbox{otherwise.}
	\end{cases}
\end{equation*}

By definition, it is easy to check that $U_\Delta^*$ is an axially symmetric slice-open set and
\begin{equation*}
	U_\Delta\cap\mathbb{R}^d=U_\Delta^+\cap\mathbb{R}^d=U_\Delta^*\cap\mathbb{R}^d=\left\{x\in\mathbb{R}^d:x\in U_\ell,\ \ell=1,...,k\right\}.
\end{equation*}

\begin{exa}\label{ex-as}
	Let $\Omega$ be an axially symmetric slice-open set and $I\in\mathcal{C}$. Choose $J=(I)$ and $U=(\Omega_I)$. Then
	\begin{equation*}
		U_\Delta^\sim=U_\Delta^*=\Omega.
	\end{equation*}
\end{exa}

\begin{prop}
	Let $J=(J_1,...,J_k)^T\in\mathcal{C}_*^k$ and $U\in\tau[J]$. Then $U_\Delta^\backsim$ is a slice-open set in $\mathcal{W}_\mathcal{C}^d$ and for each $\ell\in\{1,2,...,k\}$,
	\begin{equation}\label{eq-udp}
		U_\Delta^*\cap\mathbb{C}^d_{J_\ell}\subset U_\Delta\cap\mathbb{C}^d_{J_\ell}.
	\end{equation}

	Moreover, if $J$ is not a slice-solution of $\mathcal{C}$, then
	\begin{equation}\label{eq-udb}
		U_\Delta^\backsim \cap\mathbb{C}_I^d=\begin{cases}
			(U_\ell\backslash\mathbb{R}^d)\cup \left(U_\Delta\right)_\mathbb{R},\qquad&I=\pm J_\ell,\ \mbox{for some } \ell,\\
			\Psi_i^I(U_\mathbb{C}),\qquad& I\in\mathcal{C}_{ker}(J)\backslash\{\pm J_1,...,\pm J_k\},\\
			\Psi_i^I(U_\mathbb{C}^*),\qquad&\mbox{otherwise}.
		\end{cases}
	\end{equation}
	Otherwise, $J$ is a slice-solution of $\mathcal{C}$, then
	\begin{equation}\label{eq-udb1}
		U_\Delta^\backsim \cap\mathbb{C}_I^d=\begin{cases}
			(U_\ell\backslash\mathbb{R}^d)\cup \left(U_\Delta\right)_\mathbb{R},\qquad&I=\pm J_\ell,\ \mbox{for some } \ell,\\
			\Psi_i^I(U_\mathbb{C}),\qquad& I\in\mathcal{C}_J^+\backslash\{\pm J_1,...,\pm J_k\},\\
			\Psi_i^{-I}(U_\mathbb{C}),\qquad&\mbox{otherwise}.
		\end{cases}
	\end{equation}
\end{prop}

\begin{proof}
	\eqref{eq-udp}, \eqref{eq-udb} and \eqref{eq-udb1} hold directly by definition. Note that $(U_\ell\backslash\mathbb{R}^d)\cup \left(U_\Delta\right)_\mathbb{R}$ (if $I=\pm J_\ell$), $\Psi_i^I(U_\mathbb{C})$ and $\Psi_i^I(U_\mathbb{C}^*)$ are open in $\mathbb{C}_I^d$ for each $I\in\mathcal{C}$. Therefore by \eqref{eq-udb} and \eqref{eq-udb1}, $U_\Delta^\backsim \cap\mathbb{C}_I^d$ is open in $\mathbb{C}_I^d$ for each $I\in\mathcal{C}$. By definition, $U_\Delta^\backsim$ is a slice-open set.
\end{proof}

\begin{thm}(Extension Theorem)\label{th-et}
	$J=(J_1,...,J_k)^T\in\mathcal{C}_*^k$, and $U=(U_1,...,U_k)^T\in\tau[J]$. Assume that $f:\bigcup_{\ell=1}^k U_\ell\rightarrow\mathbb{R}^{2n}$ is a function such that for each $\ell\in\{1,...,k\}$, $f|_{U_\ell}$ is holomorphic. Then $f|_{(\bigcup_{\ell=1}^k U_\ell\backslash\mathbb{R}^d)\cup \left(U_\Delta\right)_\mathbb{R}}$ can be extended to a slice regular function $\widetilde{f}:U_\Delta^\backsim\rightarrow\mathbb{R}^{2n}$.
\end{thm}

\begin{proof}
	According to Lemma \ref{lm-lut}, for each $I\in\mathcal{C}_{ker}(J)$, we can define a holomorphic function $g[I]:\Psi_i^I(U_\mathbb{C})\rightarrow\mathbb{R}^{2n}$ by
	\begin{equation*}
		g[I](x+yI)=(1,I)\zeta^+(J)f(x+yJ),\qquad\forall\ x+yi\in U_\mathbb{C}.
	\end{equation*}
	Similarly, we note that $K:=(J_1,-J_1)\in\mathcal{C}^2_*$ and $V:=(U_1,U_1)\in\tau[K]$. Note that $K$ is a slice-solution of $\mathcal{C}$, and then $\mathcal{C}_{ker}(K)=\mathcal{C}$. Again by Lemma \ref{lm-lut}, for each $I\in\mathcal{C}$, we can define a holomorphic function {$h[I]:\Psi_i^I(U_\mathbb{C}^*)\rightarrow\mathbb{R}^{2n}$ by
	\begin{equation*}
		h[I](x+yI)=(1,I)\zeta^+(K)f(x+yK),\qquad\forall\ x+yi\in U_\mathbb{C}^*.
	\end{equation*}}
	Moreover, by definition and \eqref{eq-gg}, for each $L_1\in\mathcal{C}_{ker}(J)$ and {$L_2\in\mathcal{C}$},
	\begin{equation*}
		f=g[L_1]=h[L_2],\qquad \mbox{on} \qquad \left(U_\Delta^\backsim\right)_\mathbb{R}.
	\end{equation*}

	If $J$ is not a slice-solution. We define a function $\widetilde{f}:U_\Delta^\backsim\rightarrow\mathbb{R}^{2n}$ by
	\begin{equation}\label{eq-wfq}
		\widetilde{f}(q):=\begin{cases}
			f(q),\qquad& q\in \mathbb{C}_{J_\ell}^d,\mbox{ for some }\ell\in\{1,...,k\},\\
			g[I](q),\qquad &  q\in \mathbb{C}_{I}^d\backslash\mathbb{R}^d,\mbox{ for some } I\in\mathcal{C}_{ker}(J)\backslash\{\pm J_1,...,\pm J_k\},\\
			h[I](q),\qquad &  q\in \mathbb{C}_{I}^d\backslash\mathbb{R}^d,\mbox{ for some } I\in\mathcal{C}^+_J\backslash\mathcal{C}_{ker}(J).
		\end{cases}
	\end{equation}
	By definition, {$\widetilde{f}|_{U_\Delta^\backsim\cap\mathbb{C}_I^d}$ is holomorphic for all $I\in\mathcal{C}$}. It implies that $\widetilde{f}$ is slice regular. Since $\bigcup_{\ell=1}^k U_\ell\subset\bigcup_{\ell=1}^k \mathbb{C}^d_\ell$, it follows from \eqref{eq-wfq} that
	\begin{equation*}
		\widetilde{f}=f,\qquad\mbox{ on }\qquad U_\Delta^\backsim\cap\left(\bigcup_{\ell=1}^k U_\ell\right)=\left(\bigcup_{\ell=1}^k U_\ell\backslash\mathbb{R}\right)\cup \left(U_\Delta\right)_\mathbb{R}.
	\end{equation*}

	Otherwise, $J$ is a slice-solution. We define a function $\widetilde{f}:U_\Delta^\backsim\rightarrow\mathbb{R}^{2n}$ by
	\begin{equation*}
		\widetilde{f}(q):=\begin{cases}
			f(q),\qquad& q\in \mathbb{C}_{J_\ell}^d,\mbox{ for some }\ell\in\{1,...,k\},\\
			g[I](q),\qquad &  q\in \mathbb{C}_{I}^d\backslash\mathbb{R}^d,\mbox{ for some } I\in\mathcal{C}_J^+\backslash\{\pm J_1,...,\pm J_k\}.
		\end{cases}
	\end{equation*}
	Similarly, $\widetilde{f}$ is slice regular, and
	\begin{equation*}
		\widetilde{f}=f,\qquad\mbox{ on }\qquad U_\Delta^\backsim\cap\left(\bigcup_{\ell=1}^k U_\ell\right).
	\end{equation*}
\end{proof}

\section{Domains of slice regularity}\label{sc-dsr}

In this section, we consider domains of slice regularity for slice regular functions. These are the counterparts in this framework of the domains of holomorphy in several complex variables. Using the methods in \cite[Section 7]{Dou2020004} we generalize  some results proved in \cite[Section 9]{Dou2020001} from the case of quaternions to a general case of LSCS algebras, which includes the case of alternative algebras.  For example, for several variables, an axially symmetric slice-open set $\Omega$ is a domain of slice regularity if and only if one of its slice $\Omega_I$, $I\in\mathcal{C}$ is a domain of holomorphy in $\mathbb{C}_I^d$. We also define a generalization of $\sigma$-balls, called hyper-$\sigma$-polydiscs and we give a property of domains of slice regularity, see Proposition \ref{pr-li}. This proposition extends \cite[Proposition 9.7]{Dou2020001}

\begin{defn}\label{df-sos}
	A slice-open set $\Omega\subset\mathcal{W}_\mathcal{C}^d$ is called a domain of slice regularity if there are no slice-open sets $\Omega_1$ and $\Omega_2$ in $\mathcal{W}_\mathcal{C}^d$ with the following properties.
	\begin{enumerate}[\upshape (i)]
		\item $\varnothing\neq\Omega_1\subset\Omega_2\cap\Omega$.
		\item $\Omega_2$ is slice-connected and not contained in $\Omega$.
		\item For any slice regular function $f$ on $\Omega$, there is a slice regular function $\widetilde{f}$ on $\Omega_2$ such that $f=\widetilde{f}$ in $\Omega_1$.
	\end{enumerate}
	Moreover, if there are slice-open sets $\Omega,\Omega_1,\Omega_2$ satisfying (i)-(iii), then we call $(\Omega,\Omega_1,\Omega_2)$ a slice-triple.
\end{defn}

In a similar way, we give the following definition:
\begin{defn}
	Let {$\Omega\subset\mathcal{W}_{\mathcal{C}}^d$} be a slice-open set, $I\in\mathcal{C}$ and $U_1$, $U_2$ be open sets in $\mathbb{C}_I^d$. $(\Omega,U_1,U_2)$ is called an $I$-triple if
	\begin{enumerate}[\upshape (i)]
		\item $\varnothing\neq U_1\subset U_2\cap\Omega_I$.
		\item $U_2$ is connected in $\mathbb{C}_I^d$ and not contained in $\Omega_I$.
		\item For any slice regular function $f$ on $\Omega$, there is a holomorphic function $\widetilde{f}:U_2\rightarrow\mathbb{R}^{2n}$ such that $f=\widetilde{f}$ in $U_1$.
	\end{enumerate}
\end{defn}

\begin{lem}\label{pr-ubs}
	Let $U$ be a non-empty slice-open set and $\Omega$ be an slice-domain with $U\subsetneq \Omega$. Then $\Omega\cap\partial_I U_I\neq\varnothing$ for some $I\in\mathcal{C}$, where $\partial_I U_I$ is the boundary of $U_I$ in $\mathbb{C}_I^d$.
\end{lem}

\begin{proof}
	This proof is similar with the proof of \cite[Theorem 9.3]{Dou2020001}.
	
	Suppose that $\Omega\cap\partial_I U_I=\varnothing$ for each $I\in\mathcal{C}$. Since $\Omega_I$ and $\mathbb{C}_I^d\backslash(\partial_I U_I\cup U_I)$ are open in $\mathbb{C}_I^d$, so is
	\begin{equation*}
		\begin{split}
			\left[\Omega\cap\left(\mathcal{W}_\mathcal{C}^d\backslash U\right)\right]_I
			&=\Omega_I\cap\left(\mathbb{C}_I^d\backslash U_I\right)
			=\Omega_I\backslash(\Omega_I\cap U_I)
			\\&=\Omega_I\backslash[\Omega_I\cap (\partial_IU_I\cup U_I)]
			=\Omega_I\cap\left[\mathbb{C}_I^d\backslash(\partial_I U_I\cup U_I)\right].
		\end{split}
	\end{equation*}
	By definition, $\Omega\cap(\mathcal{W}_\mathcal{C}^d\backslash U)$ is slice-open. Hence $\Omega$ is the disjoint union of the nonempty slice-open sets $\Omega\cap(\mathcal{W}_\mathcal{C}^d\backslash U)$ and $\Omega\cap U$. It implies that $\Omega$ is not slice-connected, a contradiction.
\end{proof}

\begin{prop}\label{pr-so}
	A slice-open set $\Omega\subset\mathcal{W}_\mathcal{C}^d$ is a domain of slice regularity if and only if for any $I\in\mathcal{C}$ there are no open sets $U_1$ and $U_2$ in $\mathbb{C}_I^d$ such that $(\Omega,U_1,U_2)$ is an $I$-triple.
\end{prop}

\begin{proof}
	``$\Rightarrow$'' Let $\Omega$ be a domain of slice regularity and let us suppose, by absurd, that there is an $I$-triple  $(\Omega,V_1,V_2)$ for some $I\in\mathcal{C}$ with $V_1\cap\mathbb{R}^d=\varnothing$.

	Let $f:\Omega\rightarrow\mathbb{R}^{2n}$ be a slice regular function. By Extension Theorem \ref{th-et}, where we take $J:=(I)$ and $U:=(V_2)$, we deduce that $f|_{V_1}$ can extend to a slice regular function $\widetilde{f}$ on a slice-open set $U_\Delta^\backsim\supset V_2\supset V_1$. Let $W_1$ be a connected component of $V_1$ in $\mathbb{C}_I^d$. And let $W_2$ be a slice-connected component of $U_\Delta^\backsim$ containing $W_1$. Since $V_2\supset W_1$ is connected in $\mathbb{C}_I^d$, we have $W_2\supset V_2$. It follows from $V_2\nsubseteq\Omega_I$ that $(W_2)_I\nsubseteq\Omega_I$ and then $W_2\nsubseteq\Omega$. Thus $\varnothing\neq W_1\subset W_2\cap\Omega$, $W_2$ is slice-connected and not contained in $\Omega$. Moreover, for any slice regular function $f:\Omega\rightarrow\mathbb{R}^{2n}$, there is a slice regular function $\widetilde{f}|_{W_2}$ on $W_2$ such that $f=\widetilde{f}|_{W_2}$ on $V_2$.

	We conclude that $(\Omega,W_1,W_2)$ is a slice-triple, and $\Omega$ is not a domain of slice regularity, a contradiction.

	``$\Leftarrow$'' Now we prove the converse, i.e. a slice-open set $\Omega$ is a domain of slice regularity if for each $I\in\mathcal{C}$ there are no open sets $V_1$ and $V_2$ in $\mathbb{C}^d_I$ such that $(\Omega,V_1,V_2)$ is an $I$-triple. So we suppose that $\Omega$ is not a domain of slice regularity. Then there are slice-open sets $W_1$, $W_2$ such that $(\Omega,W_1,W_2)$ is a slice-triple. Let $U$ be a slice connected component of $\Omega\cap W_2$ with $U\cap W_1\neq\varnothing$. By the Identity Principle \ref{tm-ip}, $(\Omega,U,W_2)$ is also a slice-triple.

	We claim that for any $I\in\mathcal{C}$, $U_I$ is a union of some connected components of $\Omega_I\cap(W_2)_I$ in $\mathbb{C}_I^d$. (This follows from the general fact that if $\Sigma$ be a slice-open set and $U$ be a slice-connected component of $\Sigma$. Then for each $I\in\mathcal{C}$, $U_I$ is a union of some connected components of $\Sigma_I$). Then for any $I\in\mathcal{C}$,
	\begin{equation*}
		\partial_I U_I\subset\partial_I((W_2)_I\cap\Omega_I)\subset\partial_I((W_2)_I)\cup\partial_I(\Omega_I).
	\end{equation*}
	Since $(W_2)_I\cap\partial_I((W_2)_I)=\varnothing$, we have
	\begin{equation}\label{eq-w2}
		(W_2)_I\cap\partial_I U_I\subset\partial_I\Omega_I.
	\end{equation}

	By Lemma \ref{pr-ubs}, $(W_2)_J\cap\partial_J U_J\neq\varnothing$ for some $J\in\mathcal{C}$. Let $p\in(W_2)_J\cap\partial_J U_J$. By \eqref{eq-w2}, we have $p\in\partial_J\Omega_J$, and then $p\notin \Omega_J$. Let $W_3$ be the connected component of $(W_2)_J$ containing $p$ in $\mathbb{C}_J^d$, and $U':=U_J\cap W_3$. Since $p\in\partial_J U_J$ and $p$ is an interior point of $W_3$ in $\mathbb{C}_J^d$, we have $p\in\partial_J U'$ and then $U'\neq\varnothing$. Hence
	\begin{enumerate}[\upshape (i)]
		\item $\varnothing\neq U'\subset W_3\cap\Omega_J$.
		\item $W_3$ is connected in $\mathbb{C}_J^d$ and not contained in $\Omega_J$ (by $p\in W_3$ and $p\notin\Omega_J$).
		\item Since $(\Omega,U,W_2)$ is a slice-triple, for any slice regular function $f$ on $\Omega$, there is a slice regular function $f':W_2\rightarrow\mathbb{R}^{2n}$ such that $f=f'$ in $U$. Since $W_3\subset W_2$ and $U'\subset U$, we have $f'|_{W_3}$ is a holomorphic function such that $f=f'|_{W_3}$ on $U'$.
	\end{enumerate}
	It implies that $(\Omega,U',W_3)$ is a $J$-triple, a contradiction and the assertion follows.
\end{proof}

\begin{prop}\label{pr-aa}
	Any axially symmetric slice-open set $\Omega$ is a domain of slice regularity, if and only if $\Omega_I$ is a domain of holomorphy in $\mathbb{C}_I^d$ for some $I\in\mathcal{C}$.
\end{prop}

\begin{proof}
	``$\Leftarrow$'' Suppose that an axially symmetric slice-open set $\Omega$ is not a domain of slice regularity. By Proposition \ref{pr-so}, there is an $I$-triple $(\Omega,V_1,V_2)$ for some $I\in\mathcal{C}$. Recall $\theta^I_1\in\mathbb{R}^{2n}$ is fixed in \eqref{eq-ti}. Using the fact that $\Omega$ is axially symmetric, and the Extension Theorem \ref{th-et} where we set $J=(I)$, $U=(\Omega)$, we deduce that any holomorphic function $f:=F\theta^I_1:\Omega_I\rightarrow \mathbb{C}_I \theta^I_1\subset\mathbb{R}^{2n}$ can {  extend} to a slice regular function $\widetilde f$ defined on $\Omega=U_\Delta^\backsim$ (see Example \ref{ex-as}),  where $F:\Omega_I\rightarrow\mathbb{C}_I$ is a holomorphic function in several complex analysis.
	
	Since $(\Omega,V_1,V_2)$ is an $I$-triple, the function $f|_{V_1}=\widetilde f|_{V_1}$ can  extend to a holomorphic function $\breve f:V_2\rightarrow\mathbb{R}^{2n}$. Note that $f|_{V_1}=F|_{V_1}\theta^I_1$. According to Splitting Lemma \ref{lm-wsp1} and Identity Principle in several complex analysis, we have $\breve f=\breve F \theta^I_1$ for some holomorphic functions $\breve F:V_2\subset\mathbb{C}_I^d\rightarrow\mathbb{C}_I$ with $\breve F=F$ on $V_1$.
	
	In summary, for any holomorphic function $F:\Omega_I\rightarrow\mathbb{C}_I$, there is a holomorphic function $\breve F:V_2\rightarrow\mathbb{C}_I$ such that $F=\breve F$ on $V_1$. Thus $(\Omega_I,V_1,V_2)$ is a triple for several complex analysis, and then $\Omega_I$ is not a domain of holomorphy in $\mathbb{C}_I^d$.
	
	``$\Rightarrow$'' On the contrary, suppose that $\Omega$ is an axially symmetric slice-open set such that {for some $I\in\mathcal{C}$ $\Omega_I$} is not a domain of holomorphy in $\mathbb{C}_I^d$. Then there are domains $V_1,V_2$ in $\mathbb{C}_I^d$ such that $(\Omega_I,V_1,V_2)$ is a triple for several complex analysis, i.e. the following \eqref{it-vn}, \eqref{it-vi} and \eqref{it-fa} hold. Then we deduce that \eqref{it-fas} also holds.
	\begin{enumerate}[\upshape (i)]
		\item\label{it-vn} $\varnothing\neq V_1\subset V_2\cap\Omega_I$.
		\item\label{it-vi} $V_2$ is connected in $\mathbb{C}_I^d$ and not contained in $\Omega_I$.
		\item\label{it-fa} For any holomorphic function $g:\Omega\rightarrow\mathbb{C}_I$, there is a holomorphic function $g':V_2\rightarrow\mathbb{C}_I$ such that $g=g'$ in $V_1$.
		\item\label{it-fas} For any slice regular function $f:\Omega\rightarrow\mathbb{R}^{2n}$, it follows from Splitting Lemma \ref{lm-wsp1}, Identity Principle in several complex analysis and \eqref{it-fa} that there is a holomorphic function $f':V_2\rightarrow\mathbb{R}^{2n}$ such that $f'=f$ on $V_1$.
	\end{enumerate}
	By \eqref{it-vn}, \eqref{it-vi} and \eqref{it-fas}, $(\Omega,V_1,V_2)$ is an $I$-triple. And then by Proposition \ref{pr-so}, $\Omega$ is not a domain of slice regularity.
\end{proof}

\begin{defn}\label{df-jj}
	$J=(J_1,...,J_k)^T\in \mathcal{C}^k$. We say that $J$ is a hyper-solution of $\mathcal{C}$. If $J$ is not a slice-solution of $\mathcal{C}$ and for each $I\in\mathcal{C}\backslash\mathcal{C}_{ker}(J)$, $(J_1,...,J_k,I)^T$ is a slice-solution of $\mathcal{C}$.
\end{defn}

\begin{example}
	Let $A$ be the algebra of quaternions $\mathbb{H}\cong\mathbb{R}^4$ or octonions $\mathbb{O}\cong\mathbb{R}^8$, and
	\begin{equation*}
		\mathcal{C}_A:=\{L_I:I\in A,\ I^2=-1\}.
	\end{equation*}
	Then for each $K\in\mathcal{C}_A$, $(K)$ is a hyper-solution of $\mathcal{C}_A$.
\end{example}

\begin{defn}
	Let $q=x+yJ_1\in\mathcal{W}^d_\mathcal{C}$, { $J=(J_1,J_2,...,J_m)^T$} be a hyper-solution of $\mathcal{C}$ and $r\in\mathbb{R}_+^d$. We call
	\begin{equation*}
		\Sigma(q,r,J):=\left\{\bigcup_{K\in\mathcal{C}} \Psi_i^K\left[P(z,r)\cap P(\overline{z},r)\right]\right\}\bigcup\left\{\bigcup_{K\in\mathcal{C}_{ker}(J)} \Psi_i^{K}\left[P(z,r)\right]\right\}
	\end{equation*}
	hyper-$\sigma$-polydisc with hyper-solution $J$, center $q$ and radius $r$. Here  $P(z,r)$ is a polydisc in $\mathbb{C}^d$ with center $z=x+yi$ and radius $r$.
\end{defn}

\begin{example}
	The $\sigma$-ball in $\mathbb{H}\ (\cong L(\mathbb{H}))$ and $\mathbb{O}\ (\cong L(\mathbb{O}))$ are hyper-$\sigma$-polydiscs.
\end{example}

\begin{rmk}\label{rm-ls}
	Let $\Sigma(p,r,J)$ be a hyper-$\sigma$-polydisc with hyper-solution $J=(J_1,...,J_m)^T\in\mathcal{C}_*^m$, center $q\in\mathbb{C}_{J_1}^d$ and radius $r\in\mathbb{R}_+^d$. Take {  $U=(\Psi_i^{J_1}\left[P(z,r)\right],...,\Psi_i^{J_m}\left[P(z,r)\right])^T$}. It is easy to check by definition that
	\begin{equation*}
		\Sigma(q,r,J)=U_\Delta^\sim.
	\end{equation*}
	Therefore if we can define a function $f:\bigcup_{\ell=1}^m \Psi_i^{J_\ell}\left[P(z,r)\right]\rightarrow\mathbb{R}^{2n}$ such that $f_{J_\ell}$, $\ell=1,...,m$ are holomorphic, then $f$ can extend to a slice regular function $\widetilde{f}:\Sigma(q,r,J)\rightarrow\mathbb{R}^{2n}$.
\end{rmk}

\begin{prop}\label{pr-hs}
	Any hyper-$\sigma$-polydisc is a domain of slice regularity.
\end{prop}

\begin{proof}
	Let $\Sigma(p,r,I)$ be a hyper-$\sigma$-polydisc with hyper-solution $I=(I_1,...,I_m)^T\in\mathcal{C}_*^k$, center $p=x_0+y_0I_1\in\mathbb{C}_{I_1}^d$ and radius { $r=(r_1,...,r_d)^T\in\mathbb{R}_+^d$}. By Remark \ref{rm-lj}, we can take a non-zero element
	\begin{equation}\label{eq-ab}
		a=\begin{pmatrix}
		a_1\\a_2
		\end{pmatrix}
		\in\bigcap_{\ell=1}^k\ker(1,I_\ell)\subset\left(\mathbb{R}^{2n}\right)^2.
	\end{equation}
	It is easy to check from $(1,I_1)a=0$ that $a_1,a_2\neq 0$.
	
	Consider the function $g_1:\Sigma(p,r,I)\cap\mathbb{C}_{I_1}^d\rightarrow\mathbb{R}^{2n}$ defined by
	\begin{equation*}
		g_1(q)=\left[\sum_{\ell=1}^d\sum_{n\in\mathbb{N}}\left(\frac{q_\ell-p_\ell}{r_\ell}\right)^{2^n}\right] a_2.
	\end{equation*}
	We know that from Splitting Lemma \ref{lm-wsp1} (by taking $\xi_1=a_2$) and classical complex analysis arguments, we have that $g_1$ does not extend to a holomorphic function near any point of the boundary of $\Sigma_{I_1}(p,r,I):=\Sigma(p,r,I)\cap\mathbb{C}_{I_1}$.
	
	Obviously, there are two function $F,G:D(x_0+y_0i,r)\rightarrow\mathbb{R}$ such that
	\begin{equation*}
		F(x+yi)+G(x+yi)I_1=\sum_{\ell=1}^d\sum_{n\in\mathbb{N}}\left(\frac{q_\ell-p_\ell}{r_\ell}\right)^{2^n}.
	\end{equation*}

	By \eqref{eq-ab}, we have
	\begin{equation*}
		\begin{split}
			g_1(x+yI_1)=&[F(x+yi)+G(x+yi)I_1]a_2\\
			=&F(x+yi)a_2-G(x+yi)a_1+G(x+yi)(1,I_1)a\\
			=&F(x+yi)a_2-G(x+yi)a_1.
		\end{split}
	\end{equation*}
	
	Similarly, we can define holomorphic function {$g_\ell:\left[\Sigma(p,r,I)\right]_{I_\ell}\rightarrow\mathbb{R}^{2n}$}, $\ell=1,...,m$ by
	\begin{equation*}
		g_{\ell}(x+yI_\ell)=F(x+yi)a_2-G(x+yi)a_1.
	\end{equation*}
	 Then we can take a function {$h:\bigcup_{\ell=1}^m\left[\Sigma(p,r,I)\right]_{I_\ell}\rightarrow\mathbb{R}^{2n}$} such that { $h|_{\Sigma(p,r,I)_{I_\ell}}=g_\ell$}, {$\ell=1,...,m$}. By Remark \ref{rm-ls}, $h$ can extend to a slice regular function $f:\Sigma(p,r,I)\rightarrow\mathbb{R}^{2n}$. It is easy to check that for each $K\in\mathcal{C}_{ker}(I)$,
	\begin{equation*}
		f_K (x+yK)=F(x+yi)a_2-G(x+yi)a_1.
	\end{equation*}
	Thence $f_K$ also does not extend to a holomorphic function near any point of the boundary of $\Sigma_{K}(p,r,I)$.
	
	If $\Sigma(p,r,I)$ is not a domain of slice regularity, then by Proposition \ref{pr-so}, there a $K_0$-triple $(\Sigma(p,r,I),V_1,V_2)$ for some $K_0\in\mathcal{C}$. Let $V_1'$ be a connected component of $\Sigma(p,r,I)\cap V_2$ in $\mathbb{C}_{K_0}^d$ with $V_1'\cap V_1\neq\varnothing$. Then $(\Sigma(p,r,J),V_1',V_2)$ is also a $K_0$-triple. By the same method for \eqref{eq-w2} we have
	\begin{equation*}
		V_2\bigcap\partial_{K_0} V_1'\subset\partial_{K_0}\left(\Sigma_{K_0}(p,r,I)\right).
	\end{equation*}
	
	If $K_0\in\pm\mathcal{C}_{ker}(I)$, then the holomorphic function $f_{K_0}:\Sigma_{K_0}(p,r)\rightarrow\mathbb{R}^{2n}$ can extend to a holomorphic function near a point of the boundary of $\Sigma_{K_0}(p,r,I)$, a contradiction.
	
	Otherwise, $K_0\notin \pm\mathcal{C}_{ker}(I)$. Take $z=x'+y'L\in V_2\cap\partial_{K_0} V_1'$ with $L\in\{\pm K_0\}$ such that
	\begin{equation*}
		x'+y'I_1\in\Sigma_{I_1}(p,r,I)\qquad\mbox{and}\qquad x'-y'I_1\in\partial_{I_1}(\Sigma_{I_1}(p,r,I)).
	\end{equation*}
	There is $r'\in\mathbb{R}_+$ such that
	\begin{equation*}
		B_L(x+yL,r')\subset V_2\qquad\mbox{and}\qquad B_{I_\ell}(x+yI_\ell,r')\subset\Sigma_{I_\ell}(p,r,I),\ \ell=1,...,m.
	\end{equation*}

	Since $I$ is a hyper-solution of $\mathcal{C}$ and $L\notin\mathcal{C}_{ker}(I)$, it follows by definition that $(I_1,...,I_m,L)^T\in\mathcal{C}_*^{m+1}$ is a slice-solution. Using Extension Lemma \ref{lm-lut} where we set
	\begin{equation*}
		J:=(I_1,...,I_m,L)^T\qquad\mbox{and}\qquad U:=[P(x+yi,r)\cap P(x-yi,r)]\cup B(x+yi,r'),
	\end{equation*}
	{it follows by $-I_1\in\mathcal{C}_{ker}(J)=\mathcal{C}$ that} there is a holomorphic function $f':\Psi_i^{-I_1}(U)\rightarrow\mathbb{R}^{2n}$ with $f'=f$ on $\Sigma(p,r,I)\cap\mathbb{R}^d$. Therefore we conclude that the holomorphic function $f_{I_1}:\Sigma_{I_1}(p,r)\rightarrow\mathbb{R}^{2n}$ can extend to a holomorphic function on $B_{I_1}(x-yI_1,r_1)\cup\Sigma_{I_1}(p,r,I)$ near the point $x-yI\in\partial_{I_1}(\Sigma_{I_1}(p,r,I))$, a contradiction. It implies that $\Sigma(p,r,I)$ is a domain of slice regularity.
\end{proof}

\begin{prop}\label{pr-li}
	Let $I\in\mathcal{C}$ and $\Omega$ be a domain of slice regularity. If {$\gamma\in\mathscr{P}(\mathbb{C}^d,\Omega)$} and $J\in\mathcal{C}^m_*$ with $\gamma^{J_\ell}\subset\Omega$, then $\gamma^I\subset\Omega$ for any $I\in\mathcal{C}_{ker}(J)$.
\end{prop}

\begin{proof}
	We shall prove this  by contradiction. Suppose that $\gamma^I\not\subset \Omega$ for some $I\in\mathcal{C}_{ker}(J)$. By Lemma \ref{lm-los}, there is a domain $U$ in $\mathbb{C}^d$ containing $\gamma([0,1])$ such that
	\begin{equation*}
		\Psi_i^{J_\ell}(U)\subset\Omega,\qquad\ell=1,...,k.
	\end{equation*}
	Since $I\in\mathcal{C}_{ker}(J)$, it follows by Extension Lemma \ref{lm-lut} that there is a holomorphic function $g[I]:\Psi_i^I(U)\rightarrow\mathbb{R}^{2n}$ defined by
	\begin{equation*}
		g[I](x+yI)=(1,I)
		\zeta^+(J)
		f(x+yJ),\qquad\forall\ x+yi\in U,
	\end{equation*}
	such that $g[I]=f$ on $U_\mathbb{R}$. For each fixed $w\in U_\mathbb{R}$, it follows from Splitting Lemma \ref{lm-wsp1} that the Taylor series of holomorphic functions $g[I]_I$ and $f_I$ at the point $w$ are the same. Therefore there is sufficiently small $r\in\mathbb{R}_+$ such that $B_I(w,r)\subset\Psi_i^I(U)\cap\Omega_I$ and
	\begin{equation*}
		g[I]=f,\qquad\mbox{on}\qquad B_I(w,r).
	\end{equation*}
	It is easy to check that $(\Omega,B_I(w,r),\Psi_i^I(U))$ is an $I$-triple. It implies by Proposition \ref{pr-so} that $\Omega$ is not a domain of slice regularity, a contradiction. Therefore $\gamma^I\subset\Omega$ for each $I\in\mathcal{C}_{ker}(J)$.
\end{proof}

\section{Taylor series}\label{sc-ts}
When dealing with slice regular functions an important concept is that one of slice derivative.
In this section, we define this concept in the very general case of functions with values in the Euclidean space $\mathbb R^{2n}$ and then we prove that there is a Taylor series for weak slice regular functions over slice-cones.

\subsection{Slice derivatives}

In this subsection, we generalize the slice derivative to weak slice regular functions over slice-cones. We use the notations and definitions given in Section 5.

\begin{defn}
	Let $I\in\mathcal{C}$, $U$ be an open set in $\mathbb{C}_I^d$, $\ell\in\{1,...,d\}$ and $f:U\rightarrow\mathbb{R}^{2n}$ be real differentiable. The $(I,\ell)$-derivative $\partial_{I,\ell}f:U\rightarrow\mathbb{R}^{2n}$ of $f$ is defined by
	\begin{equation*}
		\partial_{I,\ell}f(x+yI):=\frac{1}{2}\left(\frac{\partial}{\partial x_\ell}-I\frac{\partial}{\partial y_\ell}\right)f_I(x+yI).
	\end{equation*}
	Moreover, if $f$ is $\alpha$-th real differentiable for $\alpha=(\alpha_1,...,\alpha_d)\in\mathbb{N}^d$, then the $(I,\alpha)$-derivative $f^{(I,\alpha)}:U\rightarrow\mathbb{R}^{2n}$ is defined by
	\begin{equation*}
		f^{(I,\alpha)}:=(\partial_{I,1})^{\alpha_1}\cdots(\partial_{I,d})^{\alpha_d} f.
	\end{equation*}
\end{defn}

\begin{prop}
	Let $I\in\mathcal{C}$, $U\in\tau\left(\mathbb{C}_I^d\right)$, and $f:U\rightarrow\mathbb{R}^{2n}$ be real differentiable. Then
	\begin{equation*}
		\partial_{I,\ell}f=\partial_{-I,\ell}f,\qquad \forall\ \ell\in\{1,...,d\}.
	\end{equation*}
	Moreover, if $f$ is a holomorphic map, $\ell=1,\ldots, d$, then
	\begin{equation}\label{eq-pie}
		\partial_{I,\ell}f=\frac{\partial}{\partial x_\ell} f,\qquad \forall\ \ell\in\{1,...,d\}.
	\end{equation}
\end{prop}

\begin{proof}
	(i) By direct calculation, for any $\ell\in\{1,...,d\}$ and any $x+yI\in U$,
	\begin{equation*}
		\begin{split}
			\partial_{I,\ell}f(x+yI)=&\frac{1}{2}\left(\frac{\partial}{\partial x_\ell}-I\frac{\partial}{\partial y_\ell}\right)f(x+yI)
			\\=&\frac{1}{2}\left(\frac{\partial}{\partial x_\ell}-(-I)\frac{\partial}{\partial (-y_\ell)}\right)f(x+(-y)(-I))
			\\=&\partial_{-I,\ell}f(x+yI).
		\end{split}
	\end{equation*}
	
	(ii) Suppose that $f$ is holomorphic. Then for any $\ell\in\{1,...,d\}$ and $x+yI\in U$,
	\begin{equation*}
		\frac{1}{2}\left(\frac{\partial}{\partial x_\ell}+I\frac{\partial}{\partial y_\ell}\right)f(x+yI)=0.
	\end{equation*}
As a consequence we have
	\begin{equation*}
		\partial_{I,\ell}f(x+yI)
		=\frac{\partial}{\partial x_\ell}f(x+yI)
		-\frac{1}{2}\left(\frac{\partial}{\partial x_\ell}+I\frac{\partial}{\partial y_\ell}\right)f(x+yI)
		=\frac{\partial}{\partial x_\ell}f(x+yI)
	\end{equation*}
and the assertion follows.
\end{proof}
We now introduce some terminology: the real part and imaginary parts of $\mathcal{W}_{\mathcal{C}}^d$ are defined by
\begin{equation*}
	Re\left(\mathcal{W}_{\mathcal{C}}^d\right)=\mathbb{R}^d, \qquad \qquad {Im\left(\mathcal{W}_{\mathcal{C}}^d\right)=\left\{{(t_1I,...,t_dI)^T}\in\End\left(\mathbb{R}^{2n}\right)^d:t=(t_1,...,t_d)^T\in\mathbb{R}^d,\ I\in\mathcal{C}\right\}},
\end{equation*}
respectively. By its definition, it is clear that $Im\left(\mathcal{W}_{\mathcal{C}}^d\right)$ is a cone in $\left[\End\left(\mathbb{R}^{2n}\right)\right]^d$, moreover
\begin{equation*}
	\mathcal{W}_\mathcal{C}^d=Re\left(\mathcal{W}_{\mathcal{C}}^d\right)\oplus Im\left(\mathcal{W}_{\mathcal{C}}^d\right).
\end{equation*}
For $\Omega\subset\mathcal{W}_\mathcal{C}^d$ and $q\in Im(\mathcal{W}_\mathcal{C}^d)$ let us set
\begin{equation*}
	\Omega[q]:=\Omega\cap\left[Re\left(\mathcal{W}_{\mathcal{C}}^d\right)\oplus \{q\}\right].
\end{equation*}
If $q\in\mathbb{C}_I^d$ for some $I\in\mathcal{C}$, then $Re\left(\mathcal{W}_{\mathcal{C}}^d\right)\oplus \{q\}$ is a $d$-dimensional real vector subspace of $\mathbb{C}_I^d$. If $\Omega$ is slice-open then, by definition, the set
\begin{equation*}
	\Omega[q]=\Omega_I\cap\left[Re\left(\mathcal{W}_{\mathcal{C}}^d\right)\oplus \{q\}\right]
\end{equation*}
is an open set in $Re\left(\mathcal{W}_{\mathcal{C}}^d\right)\oplus \{q\}$.

\begin{defn}
	Let $\Omega\in\tau_s(\mathcal{W}_\mathcal{C}^d)$. A function $f:\Omega\rightarrow\mathbb{R}^{2n}$ is called $\alpha$-th (resp. infinitely) real-slice differentiable, if for each $q\in\Omega$, $f|_{\Omega[q]}$ is $\alpha$-th (resp. infinitely) real differentiable, where $\alpha\in\mathbb{N}^d$.
\end{defn}

Obviously, every weak slice regular functions defined on a slice-open set in $\mathcal{W}_\mathcal{C}^d$ is also infinitely real-slice differentiable.

\begin{defn} (Slice derivatives)
	Let $\Omega\in\tau_s\left(\mathcal{W}_\mathcal{C}^d\right)$, $\ell\in\{1,...,d\}$ and $f:\Omega\rightarrow\mathbb{R}^{2n}$ be a weak slice regular function. The $\ell$-slice derivative $\partial_\ell f:\Omega\rightarrow \mathbb{R}^{2n}$ of $f$ is defined by
	\begin{equation*}
		\partial_\ell f (x+w):=\frac{\partial}{\partial x_\ell} f|_{\Omega[w]}(x),
	\end{equation*}
	where $x=(x_1,...,x_d)\in\mathbb{R}^{d}$ and $w\in Im\left(\mathcal{W}_\mathcal{C}^d\right)$.
	Let $\alpha=(\alpha_1,...,\alpha_d)\in\mathbb{N}^d$. The $\alpha$-th slice derivative $f^{(\alpha)}:\Omega\rightarrow \mathbb{R}^{2n}$ is defined by
	\begin{equation*}
		f^{(\alpha)}:=(\partial_{1})^{\alpha_1}\cdots(\partial_{d})^{\alpha_d} f.
	\end{equation*}
\end{defn}

\begin{prop}
	Let $\Omega\subset\mathcal{W}_\mathcal{C}^d$, $f:\Omega\rightarrow\mathbb{R}^{2n}$ be weak slice regular and $\alpha\in\mathbb{N}^d$. Then $f^{(\alpha)}$ is weak slice regular and
	\begin{equation}\label{eq-lfr}
		\left(f^{(\alpha)}\right)_I=\left(f_I\right)^{(I,\alpha)},\qquad\forall\ I\in\mathcal{C}.
	\end{equation}	
\end{prop}

\begin{proof}
The proof is by induction. Since $f^{(0)}=f$ is weak slice regular, then \eqref{eq-lfr} holds when $\alpha=0$. Suppose that $f^{(\beta)}$ is weak slice regular and \eqref{eq-lfr} holds when $\alpha=\beta$. We shall prove that for each $\ell\in\{1,...,d\}$, $f^{(\beta+\theta_\ell)}$ is weak slice regular and \eqref{eq-lfr} holds when $\alpha=\beta+\theta_\ell$, where
	\begin{equation*}
		\theta_\ell=\left(0_{1\times(\ell-1)},1,0,...,0\right).
	\end{equation*}
	
	Let $\ell\in\{1,...,d\}$ and $I\in\mathcal{C}$. Since $f^{(\beta)}$ is weak slice regular, $\left(f^{(\beta)}\right)_I$ is holomorphic. By induction hypothesis and recalling the notation \eqref{eq-pie}, the function
	\begin{equation*}
		\left(f^{(\beta+\theta_\ell)}\right)_I
		=\left(\frac{\partial}{\partial x_\ell}f^{(\beta)}\right)
		=\left(\frac{\partial}{\partial x_\ell}(f_I)^{
			(I,\beta)}\right)
		=(f_I)^{(I,\beta+\theta_\ell)}
	\end{equation*}
	is holomorphic. Since the choice of $I$ is arbitrary, it follows that $f^{(\beta+\theta_\ell)}$ is weak slice regular and \eqref{eq-lfr} holds when $\alpha=\beta+\theta_\ell$.
	
	By induction, for any $\alpha\in\mathbb{N}^n$, $f^{(\alpha)}$ is weak slice regular and \eqref{eq-lfr} holds.
\end{proof}

\subsection{Taylor series}

In this subsection, we shall prove a Taylor expansion on $\sigma$-polydiscs for weak slice regular functions. To this end, we consider $I\in\mathcal{C}$ and $r=(r_1,...,r_d)\in\mathbb{R}_+^d=(0,+\infty]^d$. For any $z=(z_1,...,z_d)\in\mathbb{C}_I^d$, we denote the polydisc with center $z$ and radius $r$ by
\begin{equation*}
	P_I(z,r):=\left\{w=(w_1,...,w_d)\in\mathbb{C}_I^d:|z_\imath-w_\imath|\le r_\imath,\ \imath=1,...,n\right\},
\end{equation*}
and we set
\begin{equation*}
	\widetilde{P_I}(z,r):=\{x+yJ\in\mathcal{W}_\mathcal{C}^d:J\in\mathcal{C},\ \mbox{and}\ x\pm yI\in P_I(z,r)\}\bigcup P_I(z,r).
\end{equation*}
It is easy to check that if $z\in\mathbb{C}_I^d\cap\mathbb{C}_J^d$, for some $I,J\in\mathcal{C}$, then
\begin{equation*}
	\widetilde{P_I}(z,r)=\widetilde{P_J}(z,r).
\end{equation*}
Hence, we can write $\widetilde{P}(z,r)$ instead of $\widetilde{P_I}(z,r)$, without ambiguity. We call $\widetilde{P}(z,r)$ the $\sigma$-polydisc with center $z$ and radius $r$.

\begin{prop}\label{pr-ts}
	Let $I\in\mathcal{C}$, $U$ be an open set in $\mathbb{C}_I^d$ and $f:U\rightarrow\mathbb{R}^{2d}$ be holomorphic. Then for any $z_0\in U$ and $r\in\mathbb{R}_+^d$ with $P_I(z_0,r)\subset U$, we have
	\begin{equation*}
		f(z)=\sum_{\alpha\in\mathbb{N}^d}(z-z_0)^\alpha f^{(I,\alpha)}(z_0),\qquad\forall\ z\in D_I(z_0,r).
	\end{equation*}
\end{prop}

\begin{proof}
The statement can be proved using the Splitting Lemma \ref{lm-wsp1} and the Taylor expansion in several complex variables.
\end{proof}
Let us define a (partial) order relation on $\mathbb R^d$ as follows: given $\alpha=(\alpha_1,...,\alpha_d)$ and $\beta=(\beta_1,...,\beta_d)$ in $\mathbb{R}^d$ we say that $\alpha<\beta$ if
\begin{equation*}
	\alpha_\ell<\beta_\ell,\qquad\forall\ \ell=1,...,d.
\end{equation*}

Let $\alpha\in\mathbb{N}^d$ and $p\in\mathcal{W}_\mathcal{C}^d$ and let us define the map
\begin{equation*}
	\begin{split}
		(id-p)^{*\alpha}:\quad\mathcal{W}_{\mathcal{C}}^d\ &\xlongrightarrow[\hskip1cm]{}\ \End\left(\mathbb{R}^{2n}\right),
		\\ q\ &\shortmid\!\xlongrightarrow[\hskip1cm]{}\ \sum_{0\le\beta\le\alpha}\left[
		\begin{pmatrix}
			\alpha\\\beta
		\end{pmatrix}L_q^\beta L_p^{\alpha-\beta}\right],
	\end{split}
\end{equation*}
where
\begin{equation*}
	\begin{pmatrix}
		\alpha\\\beta
	\end{pmatrix}:=\begin{pmatrix}
		\alpha_1\\\beta_1
	\end{pmatrix}\cdots\begin{pmatrix}
		\alpha_d\\\beta_d
	\end{pmatrix}
\end{equation*}
is the binomial coefficient and
\begin{equation*}	L_q^\beta:=\prod_{\ell=1}^d\left(L_{q_\ell}\right)^{\beta_\ell}=\left(L_{q_1}\right)^{\beta_1}\cdots \left(L_{q_d}\right)^{\beta_d}.
\end{equation*}
For simplicity,  below we shall write  $(q-q_0)^{*\alpha}$ instead of $(id-p)^{*\alpha}(q)$ and $(q-p)^{*\alpha}a$ instead of $[(q-p)^{*\alpha}](a)$. Moreover note that  $|\cdot|$ denotes a real linear norm in $\mathbb{R}^{2n}$, i.e. $|\lambda a|=|\lambda||a|$ for each $\lambda\in\mathbb{R}$ and $a\in\mathbb{R}^{2n}$.

\begin{prop}\label{pr-lpn}
	Let $p\in\mathcal{W}_\mathcal{C}^d$ and $a\in\mathbb{R}^{2n}$. The function defined by
	\begin{equation*}
		\begin{split}
			f:\quad\mathcal{W}_\mathcal{C}^d\ &\xlongrightarrow[\hskip1cm]{}\ \mathbb{R}^{2n},
			\\ q\ &\shortmid\!\xlongrightarrow[\hskip1cm]{}\ (q-p)^{*\alpha}a,
		\end{split}
	\end{equation*}
	is a weak slice regular function.
	Moreover, if $p\in\mathbb{C}_I^d$ and $q=x_0+y_0J\in\mathbb{C}_J^d$ for some $I,J\in\mathcal{C}$, then
	\begin{equation}\label{eq-pq}
		\left|(q-p)^{*\alpha}a\right|\le (|J||I|+1) \max_{r=x_0\pm y_0I}\big|(r-p)^{*\alpha}a\big|.
	\end{equation}
\end{prop}

\begin{proof}
	(i) Let $K\in\mathcal{C}$ and $z=x+yK\in\mathbb{C}_K^d$. For each $\beta\in\mathbb{N}^d$, $b\in\mathbb{R}^{2n}$ and $\ell=\{1,...,d\}$, we have
	\begin{equation*}
		\begin{split}
			&\frac{1}{2}\left(\frac{\partial}{\partial x_\ell}+K\frac{\partial}{\partial y_\ell}\right)(L_z^\beta b)
			\\=&\left[\frac{1}{2}\left(\frac{\partial}{\partial x_\ell}+L_K\frac{\partial}{\partial y_\ell}\right)(x_\ell+y_\ell L_k)^{\beta_\ell}\right]\left(\prod_{\jmath\neq\ell}(L_{z_\jmath})^{\beta_\jmath}b\right)=0.
		\end{split}
	\end{equation*}
	It implies that, for each $\ell\in\{1,...,d\}$,
	\begin{equation*}
		\begin{split}
			&\frac{1}{2}\left(\frac{\partial}{\partial x_\ell}+K\frac{\partial}{\partial y_\ell}\right)f_K(z)
			\\=&\frac{1}{2}\left(\frac{\partial}{\partial x_\ell}+K\frac{\partial}{\partial y_\ell}\right)(z-p)^{*\alpha}a
			\\=&\sum_{0\le\beta\le\alpha}\left[\frac{1}{2}\left(\frac{\partial}{\partial x_\ell}+L_K\frac{\partial}{\partial y_\ell}\right)
			\begin{pmatrix}
				\alpha\\\beta
			\end{pmatrix}L_z^\beta \left(L_p^{\alpha-\beta}a\right)\right]=0.
		\end{split}
	\end{equation*}
	Since the choice of $z$ and $K$ is arbitrary, $f_K$ is holomorphic and then $f$ is weak slice regular.
	
	(ii) Set $M:=\max_{r=x_0\pm y_0I}\big|(r-p)^{*\alpha}a\big|$. Since $f$ is weak slice regular and $\mathcal{W}_\mathcal{C}^d$ is axially symmetric, it follows from Proposition \ref{pr-lo} that $f$ is slice. Hence
	\begin{equation*}
		(q-p)^{*\alpha}a=f(q)=f(x+yJ)=s+JIr,
	\end{equation*}
	where
	{\begin{equation*}
		\begin{cases}
			s=\frac{1}{2}\left[f(x+yI)+f(x-yI)\right],
			\\r=-\frac{1}{2}[f(x+yI)-f(x-yI)].
		\end{cases}
	\end{equation*}}
	It is easy to check that $|r|,|s|\le M$ and
	\begin{equation*}
		|JIr|\le |J||Ir|\le|J||I||r|\le |J||I| M,
	\end{equation*}
	so \eqref{eq-pq} holds.
\end{proof}

\begin{thm} (Taylor series)\label{tm-ot}
	Let $\Omega\in\tau_s\left(\mathcal{W}_\mathcal{C}^d\right)$ and $f:\Omega\rightarrow\mathbb{R}^{2n}$ be weak slice regular. Then for each  $I\in\mathbb{S}$, $q_0\in\Omega_I$ and $r\in\mathbb{R}_+^d$ with $P_I(q_0,r)\subset\Omega$, we have
	\begin{equation}\label{eq-fqsa}
		f(q)=\sum_{\alpha\in\mathbb{N}^d}(q-q_0)^{*\alpha} f^{(\alpha)}(q_0),
	\end{equation}
 for each $q$ in the slice-connected component of  $\widetilde{P}(q_0,r)\cap\Omega$ including $q_0$.
\end{thm}

\begin{proof}
	Let $q=x+yJ\in\widetilde{P}(q_0,r)$. If $J=I$, then \eqref{eq-fqsa} holds by Proposition \ref{pr-ts}. Otherwise, $J\neq I$. By definition, $x\pm yI\in P_I(q_0,r)$. Then according to \eqref{eq-pq}, for each $\alpha\in\mathbb{N}^d$,
	\begin{equation}\label{eq-faq0}
		\begin{split}
			&\left|(q-q_0)^{*\alpha} f^{(\alpha)}(q_0)\right|
			\\\le&(|J||I|+1)\left(\left|(q_I-q_0)^{*\alpha} f^{(\alpha)}(q_0)\right|+\left|(q_{-I}-q_0)^{*\alpha} f^{(\alpha)}(q_0)\right|\right)
			\\=&(|J||I|+1)\left(\left|(q_I-q_0)^{\alpha} f^{(I,\alpha)}(q_0)\right|+\left|(q_{-I}-q_0)^{\alpha} f^{(I,\alpha)}(q_0)\right|\right),
		\end{split}
	\end{equation}
	where $q_I=x+yI$ and $q_{-I}=x-yI$. Since
	\begin{equation*}
		\sum_{\alpha\in\mathbb{N}^d}\left|(q_I-q_0)^{\alpha} f^{(I,\alpha)}(q_0)\right|\qquad\mbox{and}\qquad\sum_{\alpha\in\mathbb{N}^d}\left|(q_{-I}-q_0)^{\alpha} f^{(I,\alpha)}(q_0)\right|
	\end{equation*}
	are convergent, it follows from \eqref{eq-faq0} that
	\begin{equation*}
		\sum_{\alpha\in\mathbb{N}^d}\left|(q-q_0)^{*\alpha} f^{(\alpha)}(q_0)\right|\qquad\mbox{and}\qquad \sum_{\alpha\in\mathbb{N}^d}(q-q_0)^{*\alpha} f^{(\alpha)}(q_0)
	\end{equation*}
	are also convergent.
	
	By Proposition \ref{pr-lpn}, the function $g:\widetilde{P}(q_0,r)\rightarrow\mathbb{R}^{2n}$ defined by
	\begin{equation*}
		g(q):=\sum_{\alpha\in\mathbb{N}^d}(q-q_0)^{*\alpha} f^{(\alpha)}(q_0)
	\end{equation*}
	is a weak slice regular function. Note that $f(q)$ and $g(q)$ are weak slice regular and
	\begin{equation*}
		f=g,\qquad\mbox{on}\qquad P_I(q_0,r),
	\end{equation*}
	thus, by the Identity Principle \ref{tm-ip} we deduce $f=g$ on the slice-connected component of $\widetilde{P}(q_0,r)$ including $q_0$, i.e. \eqref{eq-fqsa} holds.
\end{proof}


\end{document}